\documentclass[12pt]{amsart}
\usepackage[utf8]{inputenc}

\usepackage[english]{babel}
\usepackage[T1]{fontenc}
\usepackage{todonotes}
\usepackage{bm}

\usepackage{enumitem}
\usepackage{amssymb,amsfonts,stmaryrd,mathrsfs}
\usepackage{amsmath,latexsym,amsthm,mathtools,bbm}
\usepackage{pgf,tikz} 
\usepackage{subcaption}
\usepackage[colorlinks,citecolor=green,linkcolor=purple]{hyperref}
\usepackage[thinlines]{easytable}
\usepackage[capitalize]{cleveref}

\newtheorem{thm}{Theorem}[section]
\newtheorem{cor}[thm]{Corollary}
\newtheorem{lem}[thm]{Lemma}
\newtheorem{prop}[thm]{Proposition}
\newtheorem{defi}[thm]{Definition}
\newtheorem{conj}{Conjecture}

\theoremstyle{remark}
\newtheorem{rmq}{Remark}

\newcommand{\bfc}{\boldsymbol{c}}
\newcommand{\bfh}{\boldsymbol{h}}
\newcommand{\bfg}{\boldsymbol{g}}
\newcommand{\Jch}{\theta^{(\alpha)}}
\newcommand{\MchJack}{{\boldsymbol{\theta}}^{(\alpha,\gamma)}}
\newcommand{\Mch}{\boldsymbol{\widetilde\theta}^{(q,t)}}

\newcommand{\Jla}{J^{(\alpha)}_{\lambda}}

\newcommand{\tH}{\widetilde{H}^{(q,t)}}
\newcommand{\J}{{J}^{(q,t)}}
\newcommand{\Jsh}{{J}^{*}}
\newcommand{\JJack}{{\mathfrak J}^{(\alpha,\gamma)}}
\newcommand{\normJ}{{j}^{(q,t)}}
\newcommand{\normJt}{\widetilde{j}^{(\alpha,\gamma)}}

\newcommand{\nablaJ}{\boldsymbol{\nabla}}
\newcommand{\DeltaJ}{\boldsymbol{\Delta}}
\newcommand{\GammaJ}{\boldsymbol{\Gamma}}

\newcommand{\Pop}{\mathcal{P}}
\newcommand{\Top}{\mathcal{T}}
\newcommand{\Gammaplus}[1]{ {  \Gamma}^{(+)}_{#1}}
\newcommand{\GammaplusJ}[1]{ \mathbf{ \Gamma}^{(+)}_{#1}}

\DeclareMathOperator{\Exp}{Exp}

\definecolor{green}{RGB}{43,92,47}
\definecolor{blue}{RGB}{40,68,104}
\definecolor{red}{RGB}{254, 113, 96}
\definecolor{purple}{RGB}{102,0,51}
\definecolor{gray}{RGB}{224,224,224}
\definecolor{lightpurple}{RGB}{255, 249, 242}
\hypersetup{colorlinks = true, urlcolor=blue}

\author[H.~Ben Dali]{Houcine Ben Dali}
\address{Université de Lorraine, CNRS, IECL, F-54000 Nancy\\
Universit\'e de Paris, CNRS, IRIF, F-75006 Paris, France}
\curraddr{Department of Mathematics, Harvard University, Cambridge, MA 02138, U.S.A.}
\email{\href{mailto:bendali@math.harvard.edu}{bendali@math.harvard.edu}}
 
\author[M.~D'Adderio]{Michele D'Adderio}
\address{Universit\`a di Pisa\\Dipartimento di Matematica\\ Largo Bruno Pontecorvo 5, 56127 Pisa\\ Italy}
\email{\href{mailto:michele.dadderio@unipi.it}{michele.dadderio@unipi.it}}

\title[Macdonald characters]{Macdonald characters from a new formula for Macdonald polynomials}

\begin{document}
\maketitle
\begin{abstract}
We introduce a new operator $\GammaJ$ on symmetric functions, which enables us to obtain a creation formula for Macdonald polynomials. This formula provides a connection between the theory of Macdonald operators initiated by Bergeron, Garsia, Haiman and Tesler, and shifted Macdonald polynomials introduced by Knop, Lassalle, Okounkov and Sahi.

We use this formula to introduce a two-parameter generalization of Jack characters, which we call Macdonald characters. Finally, we provide a change of variables in order to formulate several positivity conjectures related to these generalized characters. Our conjectures extend some important open problems on Jack polynomials, including some famous conjectures of Goulden and Jackson.
\end{abstract}

\section{Introduction}
\subsection{Jack and Macdonald polynomials}
Jack polynomials are symmetric functions depending on one parameter $\alpha$ which have been introduced by Jack \cite{Jack1970/1971}. 
The combinatorial analysis of Jack polynomials has been initiated by Stanley \cite{Stanley1989}
and a first combinatorial interpretation has been given by Knop and Sahi in terms of tableaux \cite{KnopSahi1997}. 
A second family of combinatorial objects related to Jack polynomials is given by \textit{maps}, which are roughly graphs embedded in surfaces. This connection has first been observed in the conjectures of Goulden and Jackson \cite{GouldenJackson1996} and important progress has recently been made in this direction \cite{ChapuyDolega2022,
BenDaliDolega2023} with a first ``topological expansion'' of Jack polynomials in terms of maps. 

Macdonald polynomials are symmetric polynomials introduced by Macdonald in 1989, which depend on two parameters $q$ and $t$. Jack polynomials can be obtained from Macdonald polynomials by taking an appropriate limit. Several combinatorial results on Jack polynomials have been generalized to the Macdonald case, in particular, an interpretation in terms of tableaux was established in \cite{HaglundHaimanLoehr2005}. However, no connection between Macdonald polynomials and maps is known, even conjecturally. 

While Jack polynomials were one of the original inspirations for Macdonald polynomials, the two objects have diverged a bit in recent years, with Jack polynomials being most interesting to those who study probability and maps, while Macdonald polynomials have been studied more in the context of coinvariants and related algebraic geometry (like affine Springer fibers, Hilbert schemes of points, knot invariants, etc.).

Our hope is to ``reunite'' Macdonald polynomials and Jack polynomials by showing how Macdonald polynomials are directly connected back to the work of Stanley, Goulden, Jackson, Lassalle, and others on Jack polynomials and their positivity properties. 

As a first step towards this ``reunification'', we introduce in the present article some new tools that make the parallel between the Jack and Macdonald stories more compelling.

First, we prove a creation formula (\cref{eq thm creationJqt 2,eq thm creationJqt}) for Macdonald polynomials, inspired from the one used in \cite{BenDaliDolega2023} to connect Jack polynomials to maps. Second, we use this formula to introduce a Macdonald analog of Jack characters (\cref{ssec intro characters}). Finally, we formulate a Macdonald version of some Jack conjectures, including Goulden and Jackson's Matchings-Jack and $b$-conjectures.

\subsection{Symmetric functions and plethysm}\label{ssec plethysm}

Consider the graded algebra $\Lambda=\oplus_{r\geq 0}\Lambda^{(r)}$ of symmetric functions in the alphabet $(x_1,x_2,\dots)$ with coefficients in $\mathbb{Q}(q,t)$. Let $p_\lambda$ and $h_\lambda$ denote  the power-sum and the complete symmetric functions in $(x_i)_{i\geq 1}$, respectively.
We use here a variable  $u$ to keep track of the degree of the functions, and an extra variable $v$; all the functions considered are in $\Lambda[v]\llbracket u\rrbracket$. Consider the \emph{Hall scalar product} defined by $\langle p_\mu,p_\nu\rangle =\delta_{\mu,\nu} z_\mu$, where $z_\mu$ is a numerical factor, see \cref{subsec Partitions}.
Let $f^\perp$ denote the adjoint of multiplication by $f\in \Lambda$ with respect to $\langle -, -\rangle$. 

We will use the \textit{plethystic notation}: if $E(q,t,u,v,x_1,x_2,\dots)\in\Lambda[v]\llbracket u\rrbracket$ and $f\in\Lambda$ then $f[E]$ is the image of $f$ under the algebra morphism defined by
\begin{align*}
 \Lambda[v]\llbracket u\rrbracket\longrightarrow &\Lambda[v]\llbracket u\rrbracket\\
  p_k\longmapsto &E(t^k,q^k,u^k,v^k,x_1^k,\dots)\quad \text{ for every }k\geq 1.  
\end{align*}
Set $X:=x_1+x_2+\dots$. Notice that $f[X]=f(x_1,x_2,\cdots)$ for any $f$.
Moreover, 
$$p_k\left[X\frac{1-q}{1-t}\right]=\frac{1-q^k}{1-t^k}p_k(x_1,x_2,\dots)\text{ and }p_\lambda[-X]=(-1)^{\ell(\lambda)}p_\lambda(x_1,x_2,\dots).$$

We consider the scalar product
$$\langle f[X],g[X]\rangle_{q,t}:=\left\langle f[X],g\left[X\frac{1-q}{1-t}\right]\right\rangle.$$
In particular 
\begin{equation} \label{eq:qtzlambda_def}
\langle p_\mu[X],p_\nu[X]\rangle_{q,t}=\delta_{\mu,\nu}z_\mu(q,t):=\delta_{\mu,\nu} z_\mu p_\mu\left[\frac{1-q}{1-t}\right].
\end{equation}

Finally, let $\Pop_Z$ be the operator such that 
$$\Pop_Z\cdot f[X]=\Exp[ZX]f[X],$$ i.e.\ the multiplication by the \emph{plethystic exponential} 
$$\Exp[ZX]:=\sum_{n\geq 0}h_n[ZX]=\sum_{\mu\in\mathbb Y}\frac{p_\mu[ZX]}{z_\mu},$$ (here $\mathbb{Y}$ denotes the set of integer partitions) and let
$$\Top_{Z}:=\sum_{\mu \in \mathbb{Y}}\frac{p_\mu[Z]p_\mu^\perp}{z_\mu}$$ be the translation operator, so that $\Top_Z\cdot f[X]=f[X+Z]$. Note that 
$\Pop_{Z+Z'}=\Pop_Z\cdot\Pop_{Z'}$.

\subsection{A new formula for Macdonald polynomials}
In \cite{BergeronGarsiaHaimanTesler1999}, the authors introduced a remarkable family of diagonal operators on a modified version of Macdonald polynomials. These operators, known as \emph{nabla} and \emph{delta} operators (see \cref{ssec modified Macdonald}), have a rich combinatorial structure and are closely related to the space of diagonal harmonics \cite{Haiman02, CarlssonMellit2018,HaglundRemmelWilson2018,DAdderioMellit2022}. 
We consider an analog of these operators for the integral form of Macdonald polynomials\footnote{We use boldface symbols to distinguish these operators from their relatives from Section~\ref{ssec modified Macdonald}.} (denoted $J^{(q,t)}_\lambda$); let $\nablaJ$ and $\DeltaJ_v$ be the operators on symmetric functions defined by
\begin{equation}
  \nablaJ\cdot J^{(q,t)}_\lambda\! =\! (-1)^{|\lambda|}\left(\prod_{\Box\in\lambda}q^{a'(\Box)}t^{-\ell'(\Box)}\right) J^{(q,t)}_\lambda,  \label{eq nablaJ}
\end{equation}
$$\DeltaJ_v\cdot J^{(q,t)}_\lambda\! =\! \prod_{\Box\in\lambda}\left(1-v\cdot q^{a'(\Box)}t^{-\ell'(\Box)}\right) J^{(q,t)}_\lambda,$$
where the products run over the cells of the Young diagram of  $\lambda$, and where $a'$ and $\ell'$ are defined in \cref{subsec Partitions}.

We also introduce the following operator\footnote{This operator is a close relative of the Theta operator in \cite{DAdderioRomero2023}, first introduced in \cite{DAdderioIraciVWyngaerd21}.} on $\Lambda[v]\llbracket u\rrbracket$
$$\GammaJ(u,v):=\DeltaJ_{1/v} \mathcal{P}_{\frac{uv(1-t)}{1-q}}
\DeltaJ_{1/v}^{-1}\ \ .$$ 
The polynomiality of $\GammaJ(u,v)$ in the variable $v$ is a consequence of the Pieri rule.

We can now state our new formula for Macdonald polynomials.

\begin{thm} \label{thm creationJqt}
	For any partition $\lambda=[\lambda_1,\lambda_2,\dots,\lambda_k]$, we have
	\begin{equation}\label{eq thm creationJqt}
		\GammaplusJ{\lambda_1}\GammaplusJ{\lambda_2}\cdots\GammaplusJ{\lambda_k}\cdot1=\J_{\lambda}\quad \text{ where }\quad  \GammaplusJ{m}:=[u^m]\nablaJ^{-1}\GammaJ(u,q^m)\nablaJ.
	\end{equation}        
\end{thm}

It turns out that \cref{thm creationJqt} is an easy consequence of the following \emph{creation formula}.

\begin{thm}\label{thm creationJqt 2}
For any partition $\lambda=[\lambda_1,\lambda_2,\dots,\lambda_k]$, we have
 \begin{align}\label{eq thm creationJqt 2}
\GammaJ(u,q^{\lambda_1})\GammaJ(t^{-1}u,q^{\lambda_2})\cdots \GammaJ(t^{-(k-1)}u,q^{\lambda_k})\cdot 1 & =t^{-n(\lambda)}\nablaJ \Top_{\frac{1}{u(1-t)}} \J_\lambda[uX],
	\end{align}
\end{thm}
\noindent 
where $n(-)$ is  a statistic on Young diagrams, see \cref{sec preliminaries}.
\noindent In \cref{sec Proof}, we prove analogous results for modified Macdonald polynomials (cf.\ \cref{thm creation modified 2}) from which we deduce Theorems \ref{thm creationJqt 2} and \ref{thm creationJqt}.

In addition to giving a direct construction of Macdonald polynomials, \cref{thm creationJqt 2} provides a dual approach to study the structure of these polynomials. Indeed, thanks to \cref{eq thm creationJqt 2} we can think of $\J_\lambda$ as a function in the partition $\lambda$ described by the alphabet $(q^{\lambda_1},q^{\lambda_2},\dots)$. This dual approach plays a key role in this paper and is used in \cref{sec characters} to introduce a $q,t$-deformation of the characters of the symmetric group.

\subsection{Shifted symmetric functions and Macdonald characters}\label{ssec intro characters}
Kerov and Olshanski have introduced in \cite{KerovOlshanski1994} a new approach to study the asymptotics of the characters of the symmetric group, in which the characters are thought of functions in Young diagrams. These functions are known to have a structure of \textit{shifted symmetric functions}.

This approach has been generalized to the Jack case by Lassalle, who introduced \textit{Jack characters} in \cite{Lassalle2008b}. The latter are directly related to the coefficients of Jack polynomials in the power-sum basis and have been useful to understand asymptotic behavior of large Young diagrams sampled with respect to a Jack deformed Plancherel measure \cite{CuencaDolegaMoll2023,DolegaFeray2016}.  Recently, a combinatorial interpretation of Jack characters in terms of maps on non orientable surfaces has been proved in \cite{BenDaliDolega2023}, answering a positivity conjecture of Lassalle. 

We extend here this investigation by introducing \textit{Macdonald characters}. We start by recalling the definition of $(q,t)$-shifted symmetric polynomials, due to Okounkov \cite{Okounkov1998}.
\begin{defi}\label{def:shiftedsymmetric}
    We say that a polynomial in $k$ variables $f(v_1,\dots,v_k)$ is (q,t)-shifted symmetric (or simply shifted symmetric) if it is symmetric in the variables $v_1,v_2t^{-1},\dots,v_k t^{1-k}.$ 
    
    A shifted symmetric function $f(v_1,v_2,\dots)$ is a sequence $(f_k)_{k\geq1}$ of polynomials of bounded degrees, such that for each $k$ the function $f_k$ is a shifted symmetric polynomial in $k$ variables and which satisfy the following compatibility property
$$f_{k+1}(v_1,\dots,v_k,1)=f_{k}(v_1,\dots,v_k).$$
\end{defi}
We now use the operator $\GammaJ$ to introduce a two parameter deformation $\Mch_\mu$ of the characters of the symmetric group. 
\begin{defi}
Fix a partition $\mu$ and an integer $k\geq 1$. The Macdonald character with $k$ variables associated to $\mu$  is the function $\Mch_{\mu,k}(v_1,v_2,\dots,v_k)$ defined by 
\begin{align}\label{eq def characters}
    \Mch_{\mu,k}(v_1,v_2,\dots ):
    =\left\langle p_\mu,\GammaJ(1,v_1)\GammaJ(t^{-1},v_2)\cdots \GammaJ(t^{-{k-1}},v_k) \cdot 1\right\rangle_{q,t}
\end{align}    
\end{defi}

It turns out that these characters have a structure of shifted symmetric functions.
\begin{thm}\label{thm Macdonald characters}
Fix a partition $\mu$. For any $k\geq 1$, the character $\Mch_{\mu,k}$ is a shifted symmetric polynomial.
Moreover, the sequence $(\Mch_{\mu,k})_{k\geq 1}$ defines a shifted symmetric function $\Mch_\mu$, which will be called the \emph{Macdonald character} associated to the partition $\mu$.
\end{thm}

Taking an appropriate limit (cf.\ \cref{prop Jack characters}), one can recover Jack characters from Macdonald characters, and hence also the characters of the symmetric group.

It follows from \cref{thm creation modified 2} that Macdonald characters are directly related to the expansion of Macdonald polynomials in the power-sum basis. We use the creation formula of \cref{thm creationJqt 2} to deduce properties of the Macdonald characters which generalize results known in the Jack case. 

In particular, we prove that they form a basis of the space of shifted symmetric functions which lead us to introduce their structure coefficients $\bfg^\pi_{\mu,\nu}$, see \cref{cor char basis} and \cref{eq structure coefficients}.
We also make a connection between Macdonald characters and shifted Macdonald polynomials introduced in \cite{Lassalle1998,Okounkov1998}, see Equations \eqref{eq isomorphism *} and \eqref{eq Mch-p_mu^*}. 

Furthermore, Macdonald characters and their structure coefficients seem to have interesting positivity properties which we investigate by introducing a new parametrization for Macdonald polynomials.

\subsection{A new parametrization and Macdonald version of some Jack conjectures}
In \cref{sec Jack polynomials} we introduce a new parametrization $(\alpha,\gamma)$ for Macdonald polynomials which is related to Jack polynomials, see \cref{eq parametrization}. We show that this parametrization gives a natural way to give a Macdonald version of some famous conjectures about Jack polynomials.  In particular, we formulate two positivity conjectures about Macdonald characters $\Mch_\mu$ (see \cref{conj lassale macdonald}) and their structure coefficients (see \cref{conj structure coeff}). These conjectures suggest that the characters $\Mch_\mu$ have a combinatorial structure  which generalizes the one given by maps and that we hope to investigate in future works. 
We also provide a Macdonald generalization of Stanley's conjecture about the structure coefficients of Jack polynomials, see \cref{conj Stanley-Macdonald}.

In \cite{GouldenJackson1996}, Goulden and Jackson introduced two conjectures which suggest that two  families of coefficients, $c^\pi_{\mu,\nu}(\alpha)$ and $h^\pi_{\mu,\nu}(\alpha)$, obtained from the expansion of some Jack series satisfy integrality and positivity properties. These conjectures,  known as the \textit{Matching-Jack conjecture} and the \textit{$b$-conjecture}, have also a combinatorial interpretation related to counting weighted maps on non-orientable surfaces.

The Matching-Jack and the $b$-conjectures are still open despite many partial results \cite{DolegaFeray2016,DolegaFeray2017,ChapuyDolega2022,BenDali2022a,BenDali2023}. These works involve various techniques including representation theory, random matrices and differential equations.

In \cref{sec GJ conjectures}, we introduce two families of coefficients $\bfc^\pi_{\mu,\nu}(\alpha,\gamma)$ and $\bfh^\pi_{\mu,\nu}(\alpha,\gamma)$ generalizing the coefficients of the Matchings-Jack and the $b$-conjecture. These coefficients are obtained from the expansion of some Macdonald series with the parametrization $(\alpha,\gamma)$, see \cref{eq def c,eq def h}.  

It turns out that the coefficients $\bfc^\pi_{\mu,\nu}$ are a special case of the structure coefficients of Macdonald characters $\Mch_\mu$, see \cref{prop g c}.
We also establish in \cref{prop coef c-supernabla} a connection between these coefficients and the \emph{super nabla} operator recently introduced in \cite{BergeronHaglundIraciRomero2023}.

We hope that our Macdonald generalization of these conjectures could reveal new combinatorial structures of Macdonald polynomials, in particular in connection with the enumeration of maps. Moreover, generalizing the open problems  about Jack polynomials (the $b$-conjecture, Stanley's conjecture...) to the Macdonald setting might provide a new point of view to approach these conjectures, and give the possibility to use the tools provided by the theory of Macdonald polynomials, which do not
all have interesting analogues in the Jack story.

\subsection{Outline of the paper}
The paper is organized as follows. In \cref{sec preliminaries}, we give some preliminaries and notation related to partitions and Macdonald polynomials. We prove the main theorem in \cref{sec creation formula}. We introduce Macdonald characters in \cref{sec characters}. We introduce several conjectures related to these characters in \cref{sec:generalized conjectures}. 

\subsection{Acknowledgments}

The first author was partially supported by the LUE DrEAM project of Université de Lorraine. He is also grateful to Valentin Féray and Guillaume Chapuy for many enlightening discussions.

The second author was partially supported by PRIN 2022A7L229 \emph{ALTOP}, INDAM research group GNSAGA, and ARC grant ``From algebra to combinatorics, and back''.

The authors are grateful to Hong Chen for pointing out a missing factor in Theorem~3.3 in the first version of this paper. They also thank the anonymous reviewers for valuable comments that helped improve the presentation of the paper.

\section{Preliminaries}\label{sec preliminaries}

For the results of this section we refer to \cite{DAdderioRomero2023,Macdonald1995}.
\subsection{Partitions}\label{subsec Partitions}

A \textit{partition} $\lambda=[\lambda_1,...,\lambda_\ell]$ is a weakly decreasing sequence of positive integers $\lambda_1\geq...\geq\lambda_\ell>0$. We denote by $\mathbbm{Y}$ the set of all integer partitions. The integer $\ell$ is called the \textit{length} of $\lambda$ and is denoted $\ell(\lambda)$. The size of $\lambda$ is the integer $|\lambda|:=\lambda_1+\lambda_2+...+\lambda_\ell.$ If $n$ is the \textit{size} of $\lambda$, we say that $\lambda$ is a partition of $n$ and we write $\lambda\vdash n$. The integers $\lambda_1$,...,$\lambda_\ell$ are called the \textit{parts} of $\lambda$. For $i\geq 1$, we denote $m_i(\lambda)$ the number of parts of size $i$ in $\lambda$. We then set
$$z_\lambda:=\prod_{i\geq1}m_i(\lambda)!i^{m_i(\lambda)}.$$ 
We denote by $\leq$ the \emph{dominance} partial ordering on partitions, defined by 
$$\mu\leq\lambda \iff |\mu|=|\lambda|\hspace{0.3cm} \text{ and }\hspace{0.3cm} \mu_1+...+\mu_i\leq \lambda_1+...+\lambda_i \text{ for } i\geq1,$$
where we set $\lambda_j=0$ for $j>\ell(\lambda)$.

\noindent We identify a partition  $\lambda$ with its \textit{Young diagram}\footnote{One should picture $(i,j)\in \lambda$ as being a square box in the $(i,j)$-entry of a matrix, as it is custom in the English notation of tableaux.}, defined by 
$$\lambda:=\{(i,j)\mid 1\leq i\leq \ell(\lambda),1\leq j\leq \lambda_i\}.$$
\textit{The conjugate partition} of $\lambda$, denoted $\lambda'$, is the partition associated to the Young diagram obtained by reflecting the diagram of $\lambda$ with respect to the line $j=i$:
$$\lambda':=\{(i,j)\mid 1\leq j\leq \ell(\lambda),1\leq i\leq \lambda_j\}.$$
Fix a cell $\Box:=(i,j)\in\lambda$. Its \textit{arm} and \textit{leg} are respectively given by 
$$a_\lambda(\Box):=|\{(i,c)\in\lambda,c>j\}|=\lambda_i-j, \text{ and } \ell_\lambda(\Box):=|\{(r,j)\in\lambda,r>i\}|=(\lambda')_j-i.$$
Similarly, the \textit{co-arm} and \textit{co-leg} are defined by 
$$a'_\lambda(\Box):=|\{(i,c)\in\lambda,c<j\}|=j-1, \text{ and } \ell'_\lambda(\Box):=|\{(r,j)\in\lambda,r<i\}|=i-1.$$

Finally, let the  statistic $n$ on Young diagram 
\begin{equation}\label{eq n}
  n(\lambda):=\sum_{1\leq i\leq \ell(\lambda)}\lambda_{i}(i-1)=\sum_{\Box\in\lambda}\ell_\lambda'(\Box).  
\end{equation}

\subsection{Integral form of Macdonald polynomials:}\label{ssec integral form}

For the results in this section we refer to \cite[Chapter VI]{Macdonald1995} and \cite{GarsiaTeslerAdvances96}.

Macdonald has established the following characterization theorem for Macdonald  polynomials.
\begin{thm}
The Macdonald polynomials $\J_\lambda$ are the unique family of symmetric functions such that
\begin{itemize}
    \item For any $\lambda$,
    $$\J_\lambda=\sum_{\mu\leq \lambda}v^\lambda_\mu m_\mu,$$
for some coefficients $v^\lambda_\mu$.

\item For any partitions $\lambda$ and $\rho$, $$\left\langle\J_\lambda,\J_\rho\right\rangle_{q,t}=\delta_{\lambda,\rho}\normJ_\lambda, $$
where 
\begin{equation}\label{eq j norm}
    \normJ_\lambda:=\prod_{\Box\in\lambda}\left(1-q^{a_\lambda(\Box)+1}t^{\ell_\lambda(\Box)}\right)\left(1-q^{ a_\lambda(\Box)}t^{\ell_\lambda(\Box)+1}\right)
\end{equation}
\end{itemize}
Moreover, for every $r\in \mathbb{N}$ the set $\{\J_\lambda\mid \lambda\vdash r\}$ is a basis of $\Lambda^{(r)}$.
\end{thm}

If $Y:=y_1+y_2+\dots$ is a second alphabet of variables, then \textit{Cauchy identity} for Macdonald polynomials reads (cf.\ \cref{eq:qtzlambda_def})
\begin{equation}\label{eq Cauchy}
    \sum_{\lambda\vdash m}\frac{\J_\lambda[X]\J_\lambda[Y]}{\normJ_\lambda}=\sum_{\pi\vdash m}\frac{p_\pi[X]p_\pi[Y]}{z_\pi(q,t)}, \text{ for any $m\geq 0$.}
\end{equation}

\noindent Moreover, we have the following plethystic substitution formula for Macdonald polynomials.
\begin{thm}\label{thm plethystic subt}
    $$\J_\lambda\left[\frac{1-v}{1-t}\right]=\prod_{\Box\in\lambda}\left(t^{\ell_\lambda'(\Box)}-vq^{a'_\lambda(\Box)}\right)$$
\end{thm}

Let $\lambda$ be a partition. We write $\lambda\subseteq\xi$ if the diagram $\lambda$ is contained in the diagram of $\xi$. Moreover, we say that $\xi/\lambda$ is a \emph{horizontal strip}  if $\lambda\subseteq \xi$ and in each column there is at most one cell in $\xi$ and not in $\lambda$. In other terms, for every $i\geq 1$
$$\lambda'_i\leq \xi'_i\leq \lambda'_i+1.$$

\begin{thm}[Pieri rule]\label{thm Pieri rule}
Let $\lambda$ be a partition and $k\geq 1$. Then, 
$$h_k\left[X\frac{1-t}{1-q}\right]\cdot \J_\lambda=\sum_{\xi}\eta_{\lambda,\xi} \J_\xi,$$
for some coefficients $\eta_{\lambda,\xi}$,  where the sum is taken over partitions $\xi$ such that $\xi/ \lambda$ is a horizontal strip of size $k$.
    More generally, if $f$ is a symmetric function of degree $k$ then 
    $$f\cdot \J_\lambda=\sum_{\lambda\subset_k \xi}d^f_{\lambda,\xi} \J_\xi,$$
    for some coefficients $d^f_{\lambda,\xi}$, where the sum is taken over the partitions $\xi$ obtained by adding $k$ cells to $\lambda$. 
\end{thm}

\section{A new creation formula for Macdonald polynomials}\label{sec creation formula}
In this section we prove two creation formulas for modified Macdonald polynomials (Theorems \ref{thm creation modified} and \ref{thm creation modified 2}) from which we deduce the creation formulas for the integral forms stated in the introduction.
\subsection{Modified Macdonald polynomials}\label{ssec modified Macdonald}
In \cite{GarsiaHaiman1993}, Garsia and Haiman introduced a \emph{modified} version of \emph{Macdonald polynomials} 
$$\tH_\lambda=t^{n(\lambda)}J_\lambda^{(q,1/t)}\left[\frac{X}{1-1/t}\right].$$
The operators $\nabla$ and $\Delta_v$  are respectively defined by
\begin{equation}
  \nabla \tH_{\lambda}:=(-1)^{|\lambda|}\prod_{\Box\in\lambda}q^{a_\lambda'(\Box)}t^{\ell_\lambda'(\Box)}\tH_\lambda,  \label{eq def nabla}
\end{equation}
\begin{equation}
\qquad \Delta_{v} \tH_{\lambda}:=\prod_{\Box\in\lambda}\left(1-vq^{a_\lambda'(\Box)}t^{\ell_\lambda'(\Box)}\right)\tH_\lambda.\label{eq def delta}    
\end{equation}
These operators are related by the \textit{five-term relation} of Garsia and Mellit \cite{GarsiaMellit2019}
\begin{equation}\label{eq:fiveterm}
\nabla \Pop_{\frac{u}{M}}\nabla^{-1}\Pop_{\frac{uv}{M}}=\Delta_{1/v}\Pop_{\frac{uv}{M}}\Delta_{1/v}^{-1}\ ,
\end{equation}
where $M:=(1-q)(1-t)$. Let 
$$B_\lambda:=\sum_{\Box \in\lambda}q^{a'_\lambda(\Box)}t^{\ell'_\lambda(\Box)}=\sum_{1\leq i\leq \ell(\lambda)}t^{i-1}\frac{1-q^{\lambda_i}}{1-q},$$ and $D_\lambda:=M B_\lambda-1.$ We state another fundamental identity for Macdonald polynomials, due to Garsia, Haiman and Tesler \cite{GarsiaHaimanTesler2001}: for any partition $\lambda$
\begin{equation} \label{thm Tesler}
\nabla \Pop_{-\frac{u}{M}}\Top_{\frac{1}{u}} \widetilde{H}_\lambda[uX]=\Exp\left[-\frac{uXD_\lambda}{M} \right].
\end{equation}

\begin{rmq}
    Actually, a first connection between this identity and shifted Macdonald polynomials has been made in \cite[Theorem 3.2]{GarsiaHaimanTesler2001}. In the following, we prove that this identity can be "decomposed" using the operator $\Gamma$. This reformulation is a key step to obtain the construction formula \cref{thm creationJqt}.
\end{rmq}

\subsection{Creation formula for modified Macdonald polynomials}
We start by proving a modified version of \cref{thm creationJqt 2}.
Set \[\Gamma(u,v):=\Delta_{1/v}\Pop_{\frac{uv}{1-q}}\Delta_{1/v}^{-1}.\]

\begin{rmq}
Consider the operator $\tilde{\Theta}(z;v):=\Delta_v \Pop_{-\frac{z}{M}} \Delta_v^{-1}$ introduced in \cite{DAdderioIraciVWyngaerd21}. Then this operator is related to $\Gamma$ by
\[\Gamma(u,v)=\tilde{\Theta}(uv;1/v)^{-1}\tilde{\Theta}(tuv;1/v).\]
\end{rmq}

Before proving the main theorem of this subsection, we start by establishing a second expression for the operator $\Gamma$.   
\begin{lem}\label{lem Gamma}
We have
$$\Gamma(u,v)=\nabla \Pop_{\frac{u}{M}}\nabla^{-1}\Pop_{\frac{uv}{1-q}}\nabla \Pop_{\frac{-tu}{M}}\nabla^{-1}.$$
\begin{proof}
    The operator $\Gamma$ can be rewritten as follows 
$$\Gamma(u,v)=\left(\Delta_{1/v}\Pop_{\frac{uv}{M}}\Delta_{1/v}^{-1}\right)\left(\Delta_{1/v}\Pop_{\frac{-tuv}{M}}\Delta_{1/v}^{-1}\right)=\left(\Delta_{1/v}\Pop_{\frac{uv}{M}}\Delta_{1/v}^{-1}\right)\left(\Delta_{1/v}\Pop_{\frac{tuv}{M}}\Delta_{1/v}^{-1}\right)^{-1}.$$
Using the five-term relation \eqref{eq:fiveterm} on each one of the two factors, we obtain 
\begin{equation*}
  \Gamma(u,v)=\nabla \Pop_{\frac{u}{M}}\nabla^{-1}\Pop_{\frac{uv}{M}}\Pop_{\frac{-utv}{M}}\nabla \Pop_{\frac{-tu}{M}}\nabla^{-1}=\nabla \Pop_{\frac{u}{M}}\nabla^{-1}\Pop_{\frac{uv}{1-q}}\nabla \Pop_{\frac{-tu}{M}}\nabla^{-1}.\qedhere  
\end{equation*}
\end{proof}
\end{lem}

\noindent We now prove the first creation formula for modified Macdonald polynomials.

\begin{thm}\label{thm creation modified 2}
	For any partition $\lambda=[\lambda_1,\lambda_2,\dots,\lambda_{k}]$, we have
	\begin{align}\label{eq creation modified 2}
\Gamma(u,q^{\lambda_1})\Gamma(tu,q^{\lambda_2})\cdots \Gamma(t^{k-1}u,q^{\lambda_{k}})\cdot 1 & =\nabla \Top_{\frac{1}{u}} \widetilde{H}_\lambda[uX]=\nabla \widetilde{H}_\lambda[uX+1].
	\end{align}
\end{thm}

\begin{proof}
It follows, using \cref{lem Gamma}, that 
\begin{align*}
\Gamma(u,v_1)\Gamma(tu,v_2)\cdots \Gamma(t^{k-1}u,v_k) \cdot 1 
& =\nabla \Pop_{\frac{u}{M}}\nabla^{-1}\Pop_{\frac{uv_1}{1-q}}\Pop_{\frac{utv_2}{1-q}}\cdots\Pop_{\frac{ut^{k-1}v_k}{1-q}}\nabla \Pop_{\frac{-ut^{k}}{M}}\nabla^{-1}\cdot 1\ \ .
\end{align*}

\noindent Using $\nabla^{-1}\cdot 1=1$ and $\nabla \Pop_{-\frac{z}{M}}\cdot 1=\Pop_{\frac{z}{M}}\cdot 1$ (see e.g. \cite[Eq. (1.47)]{DAdderioRomero2023} with $k=n$), we get
\begin{align}
\Gamma(u,v_1)\Gamma(tu,v_2)\cdots\Gamma(t^{k-1}u,v_k)  \cdot 1 
&=\! \nabla \Pop_{\frac{u}{M}}\nabla^{-1}\Pop_{\frac{uv_1}{1-q}}\Pop_{\frac{utv_2}{1-q}}\cdots\Pop_{\frac{ut^{k-1}v_k}{1-q}}\Pop_{\frac{ut^{k}}{M}}\cdot 1\nonumber\\
&=\!\nabla \Pop_{\frac{u}{M}}\! \nabla^{-1}\Exp\!\! \left[\frac{ut^{k}X}{M}+\frac{uX}{1-q}\sum_{1\leq i\leq k}t^{i-1}v_i\right]\label{eq Gamma}\\
& =\! \nabla \Pop_{\frac{u}{M}}\! \nabla^{-1}\Exp\!\! \left[\frac{uX}{M}-\frac{uX}{M}(1-t)\sum_{1\leq i\leq k}t^{i-1}(1-v_i)\right]\!.\nonumber
\end{align}
\noindent    Fix now a partition $\lambda$.	Applying the previous equation, we get
\begin{align*}
\Gamma(u,q^{\lambda_1})\Gamma(tu,q^{\lambda_2})\cdots \Gamma(t^{\ell-1}u,q^{\lambda_{\ell}})\cdot 1 & =\! \nabla \Pop_{\frac{u}{M}}\! \nabla^{-1}\Exp\!\! \left[\frac{uX}{M}-\frac{uX}{M}(1-t)\sum_{i\geq 1}t^{i-1}(1-q^{\lambda_i})\right]\\
	& =\! \nabla \Pop_{\frac{u}{M}}\!\nabla^{-1}\Exp\!\! \left[-\frac{uXD_\lambda}{M} \right].
\end{align*}

\noindent Applying \cref{thm Tesler} concludes the proof of the theorem.
\end{proof}

\subsection{Vanishing property and second creation formula}
The purpose of this subsection is to prove a version of \cref{thm creationJqt} for modified Macdonald polynomials.
\begin{thm} \label{thm creation modified}
	For any partition $\lambda=[\lambda_1,\lambda_2,\dots,\lambda_k]$, we have
	\begin{equation}\label{eq thm creation modified}
		\Gammaplus{\lambda_1}\Gammaplus{\lambda_2}\cdots\Gammaplus{\lambda_k}\cdot1=t^{-n(\lambda)}\tH_{\lambda}
  \end{equation}
  \noindent where
  $$\Gammaplus{m}:=[u^m]\nabla^{-1}\Gamma(u,q^m)\nabla.$$
\end{thm}

We start by stating the Pieri rule for modified Macdonald polynomials, which can be deduced from \cref{thm Pieri rule} (see \cite[Proposition~2.7]{GarsiaTeslerAdvances96}). 
\begin{thm}[Pieri rule for $\tH_\lambda$]\label{thm Pieri rule modified}
Let $\lambda$ be a partition and $k\geq 1$. Then, 
$$h_k\left[\frac{X}{1-q}\right]\cdot \tH_\lambda=\sum_{\xi}\widetilde \eta_{\lambda,\xi} \tH_\xi,$$
for some coefficients $\widetilde\eta_{\lambda,\xi}$,  where the sum is taken over partitions $\xi$ such that $\xi/ \lambda$ is a horizontal strip of size $k$.
\end{thm}

Combining \cref{thm Pieri rule modified} and the definition of $\Delta_v$ (see \cref{eq def delta}), we deduce the following formula;   
\begin{equation}\label{eq Gamma formula}
  \Gamma(u,v)\cdot \tH_\lambda=\sum_{k\geq 0}u^k\sum_{\xi\vdash k+|\lambda|}\widetilde\eta_{\lambda,\xi}\tH_\xi\prod_{\Box\in\xi/ \lambda}\left(v-q^{a'(\Box)}t^{\ell'(\Box)}\right),  
\end{equation}
where the second sum is taken over partitions $\xi$ such that $\xi/\lambda$ is a horizontal strip of size $k$. We now consider for each non-negative integer $s$, the subspace of $\Lambda$ defined by 
$$\Lambda_{\leq s}:=\text{Span}_{\mathbb Q(q,t)}\left\{\tH_\lambda, \text{ for partitions }\lambda \text{ such that }\lambda_1\leq s\right\}.$$

\begin{rmq}\label{rmq Box0}
    Let $\Box_0$ be the cell of coordinates $(1,s+1)$. This cell is characterized by $a'(\Box_0)=s$ and $\ell'(\Box_0)=0$. Notice that the condition $\lambda_1\leq s$ is equivalent to saying that $\Box_0$ is not a cell of the Young diagram of $\lambda$.
\end{rmq}

We then have the following proposition.
\begin{prop}\label{prop stability and vhanishing condition}
    Fix $s,k\geq 0$ two integers. 
    \begin{enumerate}[label={\normalfont(}\arabic*\normalfont)]
        \item The space $\Lambda_{\leq s}$ is stabilized by the action of the operator $[u^k]\Gamma(u,q^s)$.
        \item If $k>s$, then $[u^k]\Gamma(u,q^s)=0$ as operators on $\Lambda_{\leq s}$. 
    \end{enumerate} 
\begin{proof}
Fix a partition $\lambda$ such that $\lambda_1\leq s$.
From \cref{eq Gamma formula}, we know that 
\begin{equation}\label{eq prop stability}
  [u^k]\Gamma(u,q^s)\cdot \tH_\lambda=\sum_{\xi\vdash k+|\lambda|}\widetilde\eta_{\lambda,\xi}\tH_\xi\prod_{\Box\in\xi/ \lambda}\left(q^s-q^{a'(\Box)}t^{\ell'(\Box)}\right),  
\end{equation}
where the sum is taken over horizontal strips $\xi/\lambda$ of size $k$. Notice that, with the notation of \cref{rmq Box0}, the quantity $q^s-q^{a'(\Box)}t^{\ell'(\Box)}$ is 0 if and only if $\Box=\Box_0$. Moreover, if $\xi$ is a partition such that $\xi_1>s$, then $\Box_0\in\xi/\lambda$ and therefore, $\xi$ does not contribute to the sum in \cref{eq prop stability}. This proves $(1)$. If $k>s$ then any horizontal strip $\xi/\lambda$ of size $k>s$ contains necessarily the cell $\Box_0$. This gives (2).    
\end{proof}
\end{prop}

We are now ready to prove \cref{thm creation modified}.
\begin{proof}[Proof of \cref{thm creation modified}]
    Let $m$ and $\ell$ denote respectively the size and the length of $\lambda$. Extracting the coefficient of $u^m$ in  \cref{eq creation modified 2}, we get 
\begin{align}
	\sum_{m_1+\dots+m_\ell=m}\left([u^{m_1}]\Gamma(u,q^{\lambda_1})\right)\cdots \left([u^{m_\ell}]\Gamma(t^{\ell-1}u,q^{\lambda_{\ell}}))\right)\cdot 1 & =\nabla \tH_\lambda[X].
\end{align}
Hence,
\begin{align}\label{eq formula modified}	\sum_{m_1+\dots+m_\ell=m}\left([u^{m_1}]\nabla^{-1}\Gamma(u,q^{\lambda_1})\nabla\right)\cdots \left([u^{m_\ell}]\nabla^{-1}\Gamma(t^{\ell-1}u,q^{\lambda_{\ell}})\nabla\right)\cdot 1 & =\tH_\lambda[X].
\end{align}
We start by proving that the only tuple $(m_1,\dots,m_\ell)$ which contributes to this sum is $(\lambda_1,\dots,\lambda_\ell).$
First, notice that for each tuple $(m_1,\dots,m_{\ell})$, we obtain by induction and using item (1) of \cref{prop stability and vhanishing condition} that for any $0\leq i \leq \ell$ we have
\begin{equation}\label{eq stability induction}
  \left([u^{m_{\ell-i}}]\nabla^{-1}\Gamma(t^{\ell-i-1}u,q^{\lambda_{\ell-i}})\nabla\right)\cdots \left([u^{m_\ell}]\nabla^{-1}\Gamma(t^{\ell-1}u,q^{\lambda_{\ell}})\nabla\right)\cdot 1\in \Lambda_{\leq \lambda_{\ell-i}}.  
\end{equation}
Moreover, if for some $1\leq i\leq \ell$, we have $m_i>\lambda_i$, then from \cref{prop stability and vhanishing condition} item (2) and \cref{eq stability induction} we have  
\begin{equation}
  \left([u^{m_{i}}]\nabla^{-1}\Gamma(t^{i-1}u,q^{\lambda_{i}})\nabla\right)\cdots \left([u^{m_\ell}]\nabla^{-1}\Gamma(t^{\ell-1}u,q^{\lambda_{\ell}})\nabla\right)\cdot 1=0.  
\end{equation}
We deduce that any tuple $(m_1,\dots,m_\ell)$ which is different from $(\lambda_1,\dots,\lambda_\ell)$ does not contribute to the sum of \cref{eq formula modified}. 
Therefore
\begin{align*}
\left([u^{\lambda_1}]\nabla^{-1}\Gamma(u,q^{\lambda_1})\nabla\right)
\left([u^{\lambda_2}]\nabla^{-1}\Gamma(tu,q^{\lambda_2})\nabla\right)\cdots \left([u^{\lambda_\ell}]\nabla^{-1}\Gamma(t^{\ell-1}u,q^{\lambda_{\ell}})\nabla\right)\cdot 1 & =\tH_\lambda[X].
\end{align*}
This can be rewritten as follows,
\begin{multline*}
\left([u^{\lambda_1}]\nabla^{-1}\Gamma(u,q^{\lambda_1})\nabla\right)
\left(t^{\lambda_2}[u^{\lambda_2}]\nabla^{-1}\Gamma(u,q^{\lambda_2})\nabla\right)\cdots \\
\left(t^{\lambda_\ell(\ell-1)}[u^{\lambda_\ell}]\nabla^{-1}\Gamma(u,q^{\lambda_{\ell}})\nabla\right)\cdot 1 =\tH_\lambda[X].
\end{multline*}
Using the definition of $n(\lambda)$ (see \cref{eq n}), we deduce that
\begin{align*}
t^{n(\lambda)}\left([u^{\lambda_1}]\nabla^{-1}\Gamma(u,q^{\lambda_1})\nabla\right)\cdots \left([u^{\lambda_\ell}]\nabla^{-1}\Gamma(u,q^{\lambda_{\ell}})\nabla\right)\cdot 1 & =\tH_\lambda[X].
\end{align*}
This completes the proof of the theorem.
\end{proof}

\begin{rmq}
    \cref{thm creation modified} can be proved independently from \cref{thm creation modified 2} using the explicit expression of the Pieri coefficients $\widetilde\eta_{\lambda,\xi}$. We prefer here the proof based on \cref{thm creation modified 2} since it is less computational and it allows to shed some light on the properties of the operator~$\Gamma$.
\end{rmq}

\subsection{Proof of Theorems \ref{thm creationJqt 2} and \ref{thm creationJqt}} \label{sec Proof}
In this subsection we deduce Theorems \ref{thm creationJqt 2} and \ref{thm creationJqt} from Theorems \ref{thm creation modified 2} and \ref{thm creation modified} respectively.

Consider the transformation $\phi$ on $\Lambda$ defined by
$$f=\sum_{\mu}d^f_\mu(q,t) p_\mu[X] \longmapsto\phi(f):=\sum_{\mu}d^f_\mu(q,1/t) p_\mu\left[\frac{X}{1-1/t}\right],$$ 
where $d^f_\mu$ are the coefficients of $f$ in the power-sum basis. Notice that $\phi$ is invertible and $$\phi^{-1}(f)=\sum_{\mu}d^f_\mu(q,1/t) p_\mu\left[X(1-t)\right] \quad \text{for any $f$.}$$ 
With this definition, one has
$$\tH_\lambda=t^{n(\lambda)}\phi(\J_\lambda)=\phi\left(t^{-n(\lambda)}\J_\lambda\right).$$
If $\mathcal{O}$ is an operator on modified Macdonald polynomials, then we define its integral $\boldsymbol{\mathcal{O}}$ version by the composition
\begin{equation}\label{eq OJ}
  \boldsymbol{\mathcal{O}}:=\phi^{-1}\cdot \mathcal{O}\cdot\phi.  
\end{equation}

In particular, 
$$\nablaJ=\phi^{-1}\cdot \nabla \cdot\phi \quad \text{ and }\quad\DeltaJ_v=\phi^{-1}\cdot \Delta_v \cdot\phi.$$
\begin{lem}
For every $i\geq 0$, we have
\[\phi^{-1}\cdot \Pop_{\frac{ut^i}{1-q}}\cdot\phi=\Pop_{\frac{t^{-i}(1-t)u}{1-q}}.\]
\end{lem}
\begin{proof}
Fix a symmetric function $F[X]=\sum_{\lambda}c_\lambda(q,t)p_\lambda$, where $c_\lambda(q,t)$ are the coefficients of the expansion of $F[X]$ in the power-sum basis and set $\widetilde F[X]:=\sum_{\lambda}c_\lambda(q,1/t)p_\lambda$. Then
\begin{align*}
& \phi\cdot \Pop_{\frac{t^{-i}(1-t)u}{1-q}}\ F[X] =\phi \Exp\left[\frac{t^{-i}(1-t)u}{1-q}X\right]F[X]\\
& = \phi \Exp\left[\frac{t^{-i}u}{1-q}X\right] \Exp\left[-\frac{t^{-i+1}u}{1-q}X\right]F[X]\\
& = \phi \sum_{k\geq 0}u^kt^{-ik}h_k\left[\frac{X}{1-q}\right] \sum_{k\geq 0}u^kt^{-(i-1)k}h_k\left[-\frac{X}{1-q}\right]F[X]\\
& = \sum_{k\geq 0}u^kt^{ik}h_k\left[\frac{X}{(1-q)(1-t^{-1})}\right] \sum_{k\geq 0}u^kt^{(i-1)k}h_k\left[-\frac{X}{(1-q)(1-t^{-1})}\right]\widetilde F\left[\frac{X}{1-t^{-1}}\right]\\
& = \Exp\left[\frac{t^{i}u}{1-q}\frac{X}{1-t^{-1}}\right] \Exp\left[-\frac{t^{i-1}u}{1-q}\frac{X}{1-t^{-1}}\right]\widetilde F\left[\frac{X}{1-t^{-1}}\right]\\
& = \Exp\left[\frac{u}{1-q}\frac{(t^i-t^{i-1})}{1-t^{-1}}X\right] \widetilde F\left[\frac{X}{1-t^{-1}}\right]\\
&=\Exp\left[\frac{ut^i}{1-q}X\right]\widetilde F\left[\frac{X}{1-t^{-1}}\right]=\Pop_{\frac{ut^i}{1-q}}\cdot\phi\ F[X]
\end{align*}
completing the proof.
\end{proof}
From the previous lemma and the observations above, we deduce that 
\begin{equation}
 \phi^{-1}\cdot \Gamma(t^iu,v)\cdot\phi=\GammaJ(t^{-i}u,v)\quad  \text{ for every }i\geq 0.   
\end{equation}

On the other hand, we have the following lemma.
\begin{lem}
The following holds
    \[\phi^{-1}\cdot \Top_{Z}\cdot\phi=\Top_{\frac{Z}{1-t}}.\]
\end{lem}
\begin{proof}
Both sides of the equation are clearly homomorphisms of $\mathbb{Q}(q,t)$-algebras, hence it is enough to check the identity on the generators $p_k[X]$: we have
\begin{align*}
	\phi^{-1}\cdot \Top_{Z}\cdot\phi\, p_k[X] & = \phi^{-1}\cdot \Top_{Z}\, p_k\left[\frac{X}{1-t^{-1}}\right]= \phi^{-1} p_k\left[\frac{X+Z}{1-t^{-1}}\right] \\ & = \phi^{-1} (1-t^{-k})^{-1}(p_k[X] + p_k[Z]) \\ & = (1-t^{k})^{-1}(p_k[X(1-t)] + p_k[Z])\\
	&  = p_k\left[\frac{X(1-t)}{1-t}\right] +p_k\left[\frac{Z}{1-t}\right]\\
	& =p_k\left[X+\frac{Z}{1-t}\right] = \Top_{\frac{Z}{1-t}}\,  p_k[X].
\end{align*}
\end{proof}

Using this lemma and the remarks above, we deduce \cref{eq thm creationJqt 2} by applying $\phi^{-1}$ on \cref{eq creation modified 2}. In a similar way, we obtain \cref{eq thm creationJqt} from \cref{eq thm creation modified}, and the following equation from \cref{eq Gamma}. 

\begin{align}\label{eq GammaJ operator}
\GammaJ(u,v_1)\GammaJ(t^{-1}u,v_2)&\cdots \GammaJ(t^{-(k-1)} u,v_k)  \cdot 1\\
=&\nablaJ \Pop_{\frac{-tu}{1-q}}\nablaJ^{-1}\Exp\left[\frac{-ut^{k-1}X}{1-q}+\frac{u(1-t)X}{1-q}\sum_{1\leq i\leq k}t^{1-i}v_i\right]\nonumber\\
=&\nablaJ \Pop_{\frac{-tu}{1-q}}\nablaJ^{-1}\Exp\left[\frac{-utX}{1-q}-\frac{u(1-t)X}{1-q}\sum_{1\leq i\leq k}t^{1-i}(1-v_i)\right].\nonumber
\end{align}

\section{Macdonald characters}\label{sec characters}

\subsection{Shifted symmetric Macdonald polynomials}
We denote by $\Lambda^*$ the algebra of shifted symmetric functions; see \cref{def:shiftedsymmetric}.  If $f$ is a shifted symmetric function, we consider its evaluation on a Young diagram $\lambda$ defined by
$$f(\lambda):=f(q^{\lambda_1},q^{\lambda_2},\dots ,q^{\lambda_k},1,1,\dots ).$$

It is well known that the space of shifted symmetric functions can be identified to a subspace of functions on Young diagrams through the map
$$f\longmapsto (f(\lambda))_{\lambda\in\mathbb Y}.$$ 

Okounkov introduced a shifted-symmetric generalization of Macdonald polynomials (see also \cite{Knop1997b,Sahi1996}).
\begin{thm}[Shifted Macdonald polynomials]\cite{Okounkov1998}\label{thm shifted Macdonald}
Let $\mu$ be a partition.  There exists a unique function  $\Jsh_\mu(v_1,v_2,\dots)$ such that 
\begin{enumerate}
    \item $\Jsh_\mu$ is shifted symmetric of degree $|\mu|$.
    \item (normalization property)
    $$\Jsh_\mu(\mu)=(-1)^{|\mu|}q^{n(\mu')}t^{-2n(\mu)}\normJ_\mu.$$
    \item (vanishing property) for any partition $\mu\not\subset\lambda$
$$\Jsh_\mu(\lambda)=0.$$
\end{enumerate}
Moreover, the top homogeneous part of $J^*_\mu$ is $\J_\mu(v_1,t^{-1}v_2,t^{-2}v_3, \dots)$.
\end{thm}
Since these polynomials are defined by their zeros, sometimes they are referred to as interpolation polynomials.

As Macdonald polynomials form a basis of $\Lambda$, using a triangularity argument we deduce that shifted Macdonald polynomials form a basis of $\Lambda^*$. As a consequence we can linearly extend the map $J^{(q,t)}_\mu\longmapsto J^*_\mu $ into an isomorphism 
\begin{equation}\label{eq isomorphism *}
  \begin{array}{ccc}
     \Lambda &\longrightarrow & \Lambda^* \\
     f &\longmapsto & f^*.   
\end{array}
\end{equation}

\begin{rmq}\label{rmq top hom}
Note that it follows from linearity that if $f$ is a homogeneous symmetric function, then the top homogeneous part of $f^*$ is equal to $f(v_1,t^{-1}v_2,\dots)$.
\end{rmq}

\subsection{An explicit isomorphism between the spaces of symmetric and shifted-symmetric functions}
The main purpose of this subsection is to give two explicit formulas for the isomorphism \cref{eq isomorphism *}. The first one, \cref{eq construction 1}, gives the image of a function $f^*$ as a shifted symmetric function  while the second formula, \cref{eq construction 2}, gives this image as a function on Young diagrams. The proof is based on \cref{eq thm creationJqt 2}.
We start with the following lemma.

\begin{lem}\label{lem GammaJ-scalar J}
    For any symmetric function $f$, and any partition $\lambda=[\lambda_1,\lambda_2,\dots,\lambda_k]$ one has
$$\left\langle f,\GammaJ(1,q^{\lambda_1})\GammaJ(t^{-1},q^{\lambda_2})\cdots \GammaJ(t^{-(k-1)},q^{\lambda_{k}}) \cdot 1\right\rangle_{q,t}=\left\langle\Pop_{\frac{1}{1-q}}\nablaJ\cdot f,t^{-n(\lambda)}\J_\lambda\right\rangle _{q,t}.$$
\begin{proof}
    From \cref{eq thm creationJqt 2}, we know that for any function $f$
    $$\left\langle f[X], \GammaJ(u,q^{\lambda_1})\GammaJ(t^{-1}u,q^{\lambda_2}) \cdots \GammaJ(t^{-(k-1)},q^{\lambda_{k}}) \cdot 1\right\rangle_{q,t} 
    =\left\langle f[X],t^{-n(\lambda)}\nablaJ \Top_{\frac{1}{u(1-t)}} \J_\lambda[uX]\right\rangle_{q,t}.$$
Since the operator $\nablaJ$ acts diagonally on the basis of Macdonald polynomials $(\J_\lambda)_{\lambda\in\mathbb Y}$ and this basis is orthogonal with respect to the scalar product $\langle-,-\rangle_{q,t}$, the operator $\nablaJ$ is self dual with this scalar product. Moreover, the dual of $\Top_{\frac{1}{u(1-t)}}$ is $\Pop_{\frac{1}{u(1-q)}}$. We conclude by specializing $u=1$.
\end{proof}
\end{lem}

We have the following theorem.

\begin{thm}\label{thm shifted-gamma}
    For any symmetric function $f$ and any $k\geq 1$, we have
    \begin{equation}\label{eq construction 1}
      f^*(v_1,\dots,v_k)=\left\langle f, \GammaJ(1,v_1)\GammaJ(t^{-1},v_2)\cdots\GammaJ(t^{-(k-1)},v_k)\cdot 1\right\rangle_{q,t}.  
    \end{equation}
    
Equivalently, for any Young diagram $\lambda$, 
\begin{equation}\label{eq construction 2}
    f^*(\lambda)=\left\langle\Pop_{\frac{1}{1-q}}\nablaJ\cdot f,t^{-n(\lambda)}\J_\lambda\right\rangle _{q,t}.
\end{equation}
    
\begin{proof}
First, notice that \cref{eq construction 1} implies \cref{eq construction 2} by \cref{lem GammaJ-scalar J}.
By definition of the isomorphism $f\mapsto f^*$, we should prove that for any partition $\lambda$ the function
\begin{equation}\label{eq RHS thm star-Gamma}
  \left\langle J_\lambda^{(q,t)}, \GammaJ(1,v_1)\GammaJ(t^{-1},v_2)\cdots\GammaJ(t^{-(k-1)},v_k)1\right\rangle_{q,t}
\end{equation}
satisfies the three properties of \cref{thm shifted Macdonald}.
First, from \cref{eq GammaJ operator}, we know that for any~$k$
$$\GammaJ(u,v_1)\cdots\GammaJ(ut^{-(k-1)},v_k)\GammaJ(ut^{-k},1)\cdot 1
=\GammaJ(u,v_1)\cdots\GammaJ(ut^{-(k-1)},v_k)\cdot 1,$$
and that for any $n,k\geq0$ the coefficient 
$$[u^n]\GammaJ(u,v_1)\cdots\GammaJ(ut^{-(k-1)},v_k)\cdot 1$$
is a homogeneous symmetric function of degree $n$ in the variables $(x_i)_{i\geq1}$, with coefficients which are shifted symmetric in $(v_i)_{1\leq i\leq k}.$
This implies that for any symmetric function $f$,  the right-hand side of \cref{eq construction 1} is a well defined shifted symmetric function in the variables $(v_i)_{1\leq i\leq k}$. All this gives property $(1)$. 

In order to obtain property $(2)$, we use \cref{lem GammaJ-scalar J} with $f=\J_\lambda$. We get that 
\begin{align*}
  \left\langle \J_\lambda,\GammaJ(1,q^{\lambda_1})\GammaJ(t^{-1},q^{\lambda_2})\cdots \GammaJ(t^{-(k-1)} ,q^{\lambda_{k}})\cdot 1\right\rangle_{q,t}
  &=\left\langle\Pop_{\frac{1}{1-q}}\nablaJ\cdot \J_\lambda,t^{-n(\lambda)}\J_\lambda\right\rangle _{q,t}\\
  &=(-1)^{|\lambda|}q^{n(\lambda')}t^{-2n(\lambda)}\normJ_\lambda.
\end{align*}
Here we used \cref{eq nablaJ}, and the fact that $[z^0]\Pop_{\frac{z}{1-q}}=1$. 
 This corresponds to property $(2)$.
Finally, for any partitions $\mu$ and $\lambda$, one has
\begin{align*}
  \left\langle \J_\mu,\GammaJ(1,q^{\lambda_1})\GammaJ(t^{-1},q^{\lambda_2})\cdots \GammaJ(t^{-(k-1)} ,q^{\lambda_{k}})\cdot 1 \right\rangle_{q,t}
  &=\left\langle\Pop_{\frac{1}{1-q}}\nablaJ\cdot \J_\mu,t^{-n(\lambda)}\J_\lambda\right\rangle _{q,t}.
\end{align*}
But from \cref{thm Pieri rule}, we know that the coefficient of $\J_\lambda$ in $\Pop_{\frac{1}{1-q}}\nablaJ\cdot \J_\mu$ is zero unless $\mu\subset \lambda$.   This proves that (\ref{eq RHS thm star-Gamma}) satisfies property $(3)$, completing the proof of the theorem.
\end{proof}
\end{thm}

\begin{rmq}
    The isomorphism given in \cref{eq construction 2} has been implicitly described by Lassalle, see \cite[Definition 1]{Lassalle1998}. However, the formula of \cref{eq construction 1} seems to be new. Note that these two formulas are complementary since \cref{eq construction 1} gives the shifted symmetry property while \cref{eq construction 2} is more suitable to prove the vanishing conditions.
\end{rmq}

\subsection{Macdonald characters are shifted symmetric}\label{ssec characters}

Recall the definition from \cref{thm Macdonald characters} of \emph{Macdonald characters}, i.e.
\begin{align*}
	\Mch_{\mu,k}(v_1,v_2,\dots ):
	=\left\langle p_\mu,\GammaJ(1,v_1)\GammaJ(t^{-1},v_2)\cdots \GammaJ(t^{-{k-1}},v_k) \cdot 1\right\rangle_{q,t}.
\end{align*}

We are now ready to prove \cref{thm Macdonald characters}.
\begin{proof}[Proof of \cref{thm Macdonald characters}]
    By definition, for any $k\geq1$ 
    \begin{align*}
      \Mch_{\mu,k}(v_1,v_2,\dots,v_k)
      &=\left\langle p_\mu,\GammaJ(1,v_1)\GammaJ(t^{-1},v_2)\cdots \GammaJ(t^{-{k-1}},v_k) \cdot 1\right\rangle_{q,t}\\
      &=p_\mu^*(v_1,v_2,\dots,v_k).
    \end{align*}
    In particular, $(\Mch_{\mu,k})_{k\geq 1}$  defines a shifted symmetric function. 
\end{proof}

We deduce the following corollary.
\begin{cor}\label{cor char basis}
The Macdonald characters $(\Mch_\mu)_{\mu\in\mathbb Y}$ form a basis of $\Lambda^*$. 
\begin{proof}
From the definition of $\Mch_\mu$ and \cref{thm shifted-gamma} we have 
\begin{align}\label{eq Mch-p_mu^*}
    \Mch_{\mu}=p_\mu^*.
\end{align}
We conclude using the fact $(p_\mu)_{\mu\in\mathbb Y}$ is basis of $\Lambda$ and that $f\longmapsto f^*$ is an isomorphism between $\Lambda$ and $\Lambda^*$. 
\end{proof}
\end{cor}

From \cref{eq construction 2}, we get that for any Young diagram $\lambda$

\begin{equation}
    \Mch_\mu(\lambda)=\left\langle\Pop_{\frac{1}{1-q}}\nablaJ\cdot p_\mu,t^{-n(\lambda)}\J_\lambda\right\rangle _{q,t}.
\end{equation}
This can be rewritten as

\begin{equation*}
     \Mch_\mu(\lambda)=\left\langle p_\mu,t^{-n(\lambda)}\nablaJ\Top_{\frac{1}{1-t}}\cdot \J_\lambda\right\rangle _{q,t}.
\end{equation*}
Hence,
\begin{equation}\label{eq characters J}
  \Mch_\mu(q^{\lambda_1},\dots q^{\lambda_k},1,\dots)=\left\{
\begin{array}{cc}
     \left\langle p_\mu, \nablaJ h^\perp_{|\lambda|-|\mu|}\left[\frac{X}{1-t}\right]\cdot t^{-n(\lambda)}J^{(q,t)}_\lambda\right\rangle_{q,t} & \text{ if } |\mu|\leq |\lambda| \\
    0  & \text{ otherwise.}
\end{array}\right.  
\end{equation}

\noindent In particular, when $|\mu|=|\lambda|$ the characters $\Mch_\mu(\lambda)$ are given by the power-sum expansion of $\J_\lambda$:
\begin{equation}\label{eq J characters}
(-1)^{|\lambda|}q^{n(\lambda')}t^{-2n(\lambda)}\J_\lambda=\sum_{\mu\vdash |\lambda|}\frac{\Mch_\mu(\lambda)}{z_\mu(q,t)} p_\mu.  \end{equation}

\subsection{Characterization theorem}
We give here a characterization theorem for $\Mch_\mu$. This characterization has been observed by Féray in the case of Jack polynomials and proved very useful in practice (see \cite{BenDaliDolega2023}). It can be seen as an analog of \cref{thm shifted Macdonald} for characters.

\begin{thm}\label{thm Feray}
Let $\mu$ be a partition. $\Mch_\mu$ is the unique function which satisfies the following properties.
\begin{enumerate}
    \item $\Mch_\mu$ is shifted symmetric of degree $|\mu|$. 
    \item $\Mch_\mu(\lambda)=0$ for any partition  $|\lambda|<|\mu|$.
    \item the top homogeneous part of $\Mch_\mu$ is $p_\mu(v_1,t^{-1}v_2,t^{-2}v_3,\dots).$
\end{enumerate}
\end{thm}

The proof is very similar.
We start by the following lemma.

\begin{lem}\label{lem uniqueness}
    Let $n\geq1$. If $G$ is shifted symmetric function of degree less or equal to $n$ with
      \begin{equation}\label{eq G van}
    G(\lambda)=0 \text{ for $|\lambda|\leq n$}. 
\end{equation}
Then $G=0$. 
    \begin{proof}
We expand $G$ in the $J^*_\xi$ basis 
\begin{equation}\label{eq G J}
  G=\sum_\xi c_\xi J^*_\xi.  
\end{equation}

As $\deg(G)\leq n,$ the sum can be restricted to partitions $\xi$ of size at most $n$. 
We will prove by contradiction that $G=0$, i.e.\ that $c_\xi=0$
for all partitions $\xi$ with $|\xi|\leq n$. Assume this is not
the case and consider a partition $\xi_0$ of minimal size such that
$c_{\xi_0}\neq 0$.
\cref{eq G van} gives $G(\xi_0)=0$ since $|\xi_0|\leq n$. On the other
hand, $J_\xi^*(\xi_0)=0$ if $\xi_0$ does not contain $\xi$ (see
property (3) of \cref{thm shifted Macdonald}). Therefore the right hand side of \cref{eq G J} evaluated on
$\xi_0$ vanishes for
all partitions $\xi$ except for $\xi=\xi_0$. Moreover,
$c_{\xi_0}\neq 0$ by the assumptions and
$J^*_{\xi_0}(\xi_0)\neq0$ from property (2) of \cref{thm shifted Macdonald}. Therefore
$G(\xi_0) = c_{\xi_0} J^{*}_{\xi_0}(\xi_0) \neq 0$, and we have reached a contradiction. Hence, $G=0$ as required.    \end{proof}
\end{lem}

We now prove the characterization theorem.
\begin{proof}[Proof of \cref{thm Feray}]
    We start by proving that $\Mch_\mu$ satisfies these three properties. The first property is given in \cref{thm Macdonald characters}, and from \cref{rmq top hom}, we know that its top homogeneous part is $p_\mu(v_1,t^{-1}v_2,t^{-2}v_3,\dots)$. Moreover, $\Mch_\mu$ is a linear combination of $J_\xi^{*}$ for $\xi$ of size $|\mu|$. This gives the vanishing property.

    Let us now prove the uniqueness. Let $F$ be a shifted symmetric function of degree $|\mu|$ with the same top degree part as $\Mch_\mu$, and such that $F(\lambda)=0$ for any $|\lambda |<|\mu|$. Set $G:=F-\Mch_\mu$. Then $G$ is a shifted symmetric function of degree at most $|\mu|-1$ with 
$G(\lambda)=0 \text{ for $|\lambda|<|\mu|$}.$ Using \cref{lem uniqueness} we get that $G=0$ hence $F=\Mch_\mu$ which gives the uniqueness.
 \end{proof}

\subsection{Positivity conjectures about the characters \texorpdfstring{$\Mch_\mu$}{}}

We conclude this section with some intriguing positivity conjectures about the operator $\GammaJ$.

\begin{conj}\label{conj Gamma-p}
    The operator $\GammaJ(z,v)$ acts positively on the basis $p_\mu\left[X\frac{1-t}{1-q}\right]$ in the variables $q',\gamma, -v, -z$. More precisely, if $\mu$ and $\nu$ are two partitions such that $|\nu|-|\mu|=n$, then
    $$(-1)^nt^{|\mu|}\left\langle[z^n]\GammaJ(z,v)\cdot p_\mu\left[X\frac{1-t}{1-q}\right],p_\nu\left[X\frac{1-q}{1-t}\right]\right\rangle$$
    is a polynomial in the variables $-v,q':=q-1,\gamma:=t-1$ with non-negative integer coefficients.
\end{conj}

This conjecture has been tested for $|\nu|\leq 9$. 
We have the following consequence of \cref{conj Gamma-p} which is closely related to a conjecture about Macdonald characters we formulate in the next section; see \cref{conj lassale macdonald}.
\begin{prop}
    \cref{conj Gamma-p} implies that $(-1)^{|\mu|}t^{2(k-1)|\mu|}\Mch_\mu(v_1,v_2,\dots,v_k)$ is a polynomial in the variables $-v_1,\cdots,-v_k$ and the parameters $q':=q-1$ and $\gamma:=t-1$ with non-negative coefficients.

    \begin{proof}
    Let us assume that \cref{conj Gamma-p} holds.
        We want to prove by induction on $k$ that for any~$n$
        $$(-1)^nt^{2(k-1)n}[z^n]\GammaJ(z,v_1)\GammaJ(t^{-1}z,v_2)\cdots \GammaJ(t^{-(k-1)}z,v_k)\cdot1$$ has 
        a positive polynomial expansion on the basis $\left(p_\mu\left[X\frac{1-t}{1-q}\right]\right)_{\mu\vdash n}.$
This would imply the claim of the proposition by the definition of Macdonald characters \cref{eq def characters}.       
    For $k=1$ this is a direct consequence of \cref{conj Gamma-p}. We assume now that the induction assumption holds for $k$. We then get that
    \begin{multline*}
      (-1)^n t^{2(k-1)n}[z^n]\GammaJ(z,v_2)\GammaJ(t^{-1}z,v_3)\cdots \GammaJ(t^{-(k-1)}z,v_{k+1})\cdot1\\  
      =(-1)^n t^{(2k-1)n}[z^n]\GammaJ(t^{-1}z,v_2)\GammaJ(t^{-2}z,v_3)\cdots \GammaJ(t^{-k}z,v_{k+1})\cdot1
    \end{multline*}
    also has positive expansion.  Applying $t^n\GammaJ(z,v_1)$ on the left and using \cref{conj Gamma-p}, we obtain that 
    $$\left((-1)^m[z^m]\GammaJ(z,v_1)\right)(-1)^nt^{2kn}[z^n]\GammaJ(t^{-1}z,v_2)\GammaJ(t^{-2}z,v_3)\cdots \GammaJ(t^{-k}z,v_{k+1})\cdot1$$
    has positive expansion on $\left(p_\mu\left[X\frac{1-t}{1-q}\right]\right)_{\mu\vdash (n+m)}.$
    In particular this is also the case for 
    $$(-1)^{m+n}t^{2k(n+m)}\left([z^m]\GammaJ(z,v_1)\right)[z^n]\GammaJ(t^{-1}z,v_2)\GammaJ(t^{-2}z,v_3)\cdots \GammaJ(t^{-k}z,v_{k+1})\cdot1.$$
    Since this holds true for any $n,m\geq 0$, we deduce the induction hypothesis for $k+1$.
    \end{proof}
\end{prop}

\section{Macdonald versions for some Jack conjectures}\label{sec:generalized conjectures}

Jack polynomials are symmetric functions which depend on a deformation parameter $\alpha$. We briefly present here some of the most important results and conjectures related to Jack polynomials. We then introduce a new change of variables which allows us to generalize these conjectures to Macdonald polynomials.

The section is organized as follows.  
In \cref{sec Jack polynomials}, we recall the definition of Jack polynomials and we introduce a new parametrization of Macdonald polynomials which is directly related to Jack polynomials. We then discuss Macdonald generalizations of Stanley's and Lassalle's conjectures in \cref{sec stanley conj,sec Lassalle conjecture} respectively. The rest of the subsections are dedicated to discuss a generalization of Goulden--Jackson's Matchings-Jack and $b$-conjectures. In \cref{sec GJ conjectures} we state the generalized conjectures. We then discuss in \cref{ssec structure coefficients} the connection of the generalized Matching-Jack conjecture to the structure coefficients of Macdonald characters and we give a reformulation of this conjecture in \cref{sec Reformulation nabla} with the super nabla operator. We finally discuss some special cases of these conjectures in \cref{sec Special cases}.
 
\subsection{Jack polynomials and a new normalization of Macdonald polynomials}\label{sec Jack polynomials}
Jack polynomials can be obtained from the integral form of Macdonald polynomials as follows (see \cite[Chapter VI, eq (10.23)]{Macdonald1995})
\begin{equation}\label{eq Jack from Macdonald}
    \lim_{t\rightarrow 1}\frac{J_\lambda^{(q=1+\alpha(t-1),t)}}{(1-t)^{|\lambda|}}=\Jla.
\end{equation}
We denote the $\langle-,-\rangle_\alpha$, the scalar product defined on power-sum functions by
$$\langle p_\mu,p_\nu\rangle_\alpha=\delta_{\mu,\nu} z_\mu \alpha^{\ell(\mu)} .$$
Jack polynomials are orthogonal with respect to this scalar product. We denote by $j^{(\alpha)}_\lambda$ their squared norm, i.e.
$$\langle J^{(\alpha)}_\lambda,J_\mu^{(\alpha)}\rangle_\alpha=\delta_{\lambda,\mu}j^{(\alpha)}_\lambda.$$

\bigskip 

We consider the following normalization of Macdonald polynomials

$$\JJack_\lambda:=\frac{\J_\lambda}{(1-t)^{|\lambda|}}_{\big|q=1+\gamma\alpha,t=1+\gamma}.$$

In the following, the parameters $(\alpha,\gamma)$ will be always related to $(q,t)$ by
\begin{equation}\label{eq parametrization}
  \left\{\begin{array}{ll}
   q=1+\gamma\alpha  \\
   t=1+\gamma 
\end{array}\right. \longleftrightarrow 
\left\{\begin{array}{ll}
   \alpha=\frac{1-q}{1-t}  \\
   \gamma=t-1.
\end{array}\right.  
\end{equation}

Note that from \cref{eq Jack from Macdonald}, we get 
\begin{equation}\label{eq Macdonald-Jack}
  \mathfrak J^{(\alpha,\gamma=0)}_\lambda=J^{(\alpha)}_\lambda.  
\end{equation}

Unlike the functions $\J_\lambda$, the normalized functions $\JJack_\lambda$ are positive in the monomial basis.

\begin{prop}
    The coefficient of $\JJack_\lambda$ in the monomial basis are polynomial in $\alpha$ and $\gamma$ with non-negative integer coefficients.
    \begin{proof}
        This can be easily obtained from the combinatorial interpretation given in \cite[Proposition 8.1]{HaglundHaimanLoehr2005} for $\J_\lambda$.
    \end{proof}
\end{prop}

Our new $(\alpha,\gamma)$-reparametrization of Macdonald polynomials will allow us to formulate Macdonald generalizations of Stanley, Lassalle and Goulden--Jackson's conjectures. In the following, we recall these conjectures in the Jack case and we then state their Macdonald generalizations.

\subsection{Stanley's conjecture}\label{sec stanley conj}

\subsubsection{Jack case}
In his seminal work \cite{Stanley1989}, Stanley formulated the following positivity conjecture about the structure coefficients of Jack polynomials (see \cite[Conjecture 8.5]{Stanley1989}).

\begin{conj}\cite{Stanley1989}
    For arbitrary partitions $\lambda$, $\mu$ and $\nu$, 
     $$\langle J^{(\alpha)}_\lambda J^{(\alpha)}_\mu,J^{(\alpha)}_\nu\rangle_{\alpha}$$
is a polynomial in $\alpha$ with non-negative integer coefficients.
\end{conj}
This conjecture is wide open, and an analog for Shifted Jack polynomials has been proposed in \cite{AlexanderssonFeray19}.
\subsubsection{Macdonald generalization}
\begin{conj}[Macdonald version of Stanley's conjecture]\label{conj Stanley-Macdonald}
    For any partitions $\lambda, \mu,\nu$ the quantity
    $$\langle\JJack_\lambda\JJack_\mu,\JJack_\nu\rangle_{q,t}$$
    is a polynomial in the parameters $\alpha$ and $\gamma$ with non-negative integer coefficients.
\end{conj}
This conjecture has been tested for $|\nu|\leq 9$. Stanley's conjecture corresponds to the case $\gamma=0$ of \cref{conj Stanley-Macdonald}. 

\subsection{Lassalle's conjecture}\label{sec Lassalle conjecture}
\subsubsection{Jack case}
Jack characters have been introduced by Lassalle \cite{Lassalle2008b} as a one parameter deformation of the characters of the symmetric group.
\begin{defi}[Jack characters]Fix a partition $\mu$. The Jack character $\theta^{(\alpha)}_\mu$ is the function on Young diagrams $\lambda$ defined by 
    $$ \theta^{(\alpha)}_\mu(\lambda):=\left\{
    \begin{array}{ll}
    \left[p_\mu\right]\frac{1}{(|\lambda|-|\mu|)!}(p_1^\perp)^{|\lambda|-|\mu|}J^{(\alpha)}_\lambda &  \mbox{ if $|\mu|\leq|\lambda|$}\\
       0 & \mbox{ if } |\lambda|<|\mu|.
        \end{array}
\right.$$
\end{defi}

Lassalle's conjecture, formulated in \cite{Lassalle2008b} and proved in \cite{BenDaliDolega2023}, suggests that the character $\Jch_\mu(\lambda)$ is a positive polynomial in $b:=\alpha-1$, and some coordinates of $\lambda$ called \textit{multirectangular coordinates}. We are here interested in a weak version of this result, which we generalize to the Macdonald case.
\begin{thm}[\cite{BenDaliDolega2023}]
Fix a partition $\mu$. The normalized Jack character $(-1)^{|\mu|}z_\mu\Jch_\mu(\lambda)$  is a polynomial in the variables $b:=\alpha-1,-\alpha \lambda_1,-\alpha \lambda_2\dots$ with non-negative integer coefficients.  
\end{thm}

\subsubsection{Macdonald generalization}

We start by introducing a normalization of Macdonald characters which is directly related to Jack characters.
\begin{equation}\label{eq def theta alpha u}
  \MchJack_\mu(s_1,s_2,\dots):=\frac{1}{\gamma^{|\mu|}z_\mu(q,t)}\Mch_\mu(1+\alpha \gamma s_1,1+\alpha \gamma s_2,\dots).  
\end{equation}
Note that $\MchJack_\mu$ is symmetric in the variables $t^{-i}s_i+\frac{t^{-i}}{\alpha \gamma}$.
For any partition $\lambda$, we denote 
\begin{equation}\label{eq def theta alpha gamma}
  \MchJack_\mu(\lambda):=\MchJack_\mu\left(\frac{q^{\lambda_1}-1}{q-1},\frac{q^{\lambda_2}-1}{q-1},\dots\right)=\frac{\Mch_\mu(\lambda)}{\gamma^{|\mu|}z_\mu(q,t)}.  
\end{equation}
Hence,
\begin{equation}\label{eq charactersJack J}
  \MchJack_\mu(\lambda)=\left\{
\begin{array}{cc}
      [p_\mu](-1)^{|\mu|} \nablaJ (1-t)^{|\lambda|-|\mu|} h^\perp_{|\lambda|-|\mu|}\left[\frac{X}{1-t}\right]\cdot t^{-n(\lambda)}\JJack_\lambda& \text{ if } |\mu|\leq |\lambda| \\
    0  & \text{ otherwise.}
\end{array}\right.  
\end{equation}

In particular, when $|\lambda|=|\mu|$ the characters  $\MchJack_\mu(\lambda)$ are given by the expansion
\begin{equation}\label{eq JJack characters}
q^{n(\lambda')}t^{-2n(\lambda)}\JJack_\lambda=\sum_{\mu\vdash |\lambda|}\MchJack_\mu(\lambda) p_\mu.  
\end{equation}

The Jack normalization of Macdonald characters are related to Jack characters by the following proposition.
\begin{prop}\label{prop Jack characters}
For any partitions $\mu$ and $\lambda$, we have
    $$\lim_{\gamma\mapsto 0}\MchJack_\mu(\lambda)=\theta^{(\alpha)}_\mu(\lambda).$$
    \begin{proof}
Since $t\rightarrow1$ as $\gamma\mapsto0$, and the operator $\nablaJ$ acts on a homogeneous functions of degree $n$ as a multiplication by $(-1)^n$. Moreover,
$$\lim_{t\rightarrow 1}h_r^\perp\left[\frac{X}{1-t}\right](1-t)^{r}=\lim_{t\rightarrow 1}(1-t)^{r}\sum_{\nu\vdash r}\prod_{i\in\nu}\frac{p_{\nu_i}^\perp}{(1-t^{\nu_i})z_\nu}=\frac{(p_1^\perp)^r}{r!}.$$
\cref{eq charactersJack J} allows to conclude.\qedhere
\end{proof}
\end{prop}

\begin{rmq}
    Actually, $\boldsymbol{\theta}_\mu^{(\alpha,\gamma=0)}$ coincides also as a polynomial in the variables $(s_i)$ with the Jack character $\theta_\mu^{(\alpha)}$. This can be shown using the previous proposition and the fact that  $\lim_{\gamma \rightarrow 0}\MchJack_\mu(s_1,s_2,\dots)$ is symmetric in the variables $s_i-i/\alpha$.
\end{rmq}

These characters seem to satisfy the following conjecture, tested for $k\leq 3$ and $|\mu|\leq 7$.

\begin{conj}\label{conj lassale macdonald}
Fix $k\geq 1$ and $\mu\in \mathbb{Y}$. Then,
    $(-1)^{|\mu|}t^{(k-1)|\mu|}z_\mu(q,t)\MchJack_\mu(s_1,s_2,\dots, s_k)$ is a polynomial in $\gamma,b:=\alpha-1, -\alpha s_1, -\alpha s_2\dots, -\alpha s_k$ with non-negative integer coefficients.
\end{conj}

\subsection{Goulden and Jackson's conjectures}\label{sec GJ conjectures}

\subsubsection{Jack case}
We start by recalling the Matching-Jack and $b$-conjectures formulated by Goulden and Jackson in \cite{GouldenJackson1996} for Jack polynomials. Let $Y:=y_1+y_2+\cdots $ and $Z:=z_1+z_2+\cdots$ be two additional alphabets of variables.
We consider the two families of coefficients $c^\lambda_{\mu,\nu}(\alpha)$ and $h^\lambda_{\mu,\nu}(\alpha)$ indexed by partitions of the same size and defined by 
\begin{align*}
    \sum_{\lambda\in\mathbb Y} u^{|\lambda|} \frac{\Jla[X]\Jla[Y]\Jla[Z]}{j^{(\alpha)}_\lambda}=\sum_{m\geq 0}\sum_{\pi,\mu,\nu\vdash m} \frac{u^m c^\pi_{\mu,\nu}(\alpha)}{z_\pi \alpha^{\ell(\pi)}}p_\pi[X]p_\mu[Y]p_\nu[Z],
\end{align*}

\begin{align*}
  \log\left(\sum_{\lambda\in\mathbb Y} u^{|\lambda|} \frac{\Jla[X]\Jla[Y]\Jla[Z]}{j^{(\alpha)}_\lambda}\right) 
  =\sum_{m\geq 0}\sum_{\pi,\mu,\nu\vdash m} \frac{u^m h^\pi_{\mu,\nu}(\alpha)}{\alpha m}p_\pi[X]p_\mu[Y]p_\nu[Z].   
\end{align*}

For $\alpha=1$ (resp. $\alpha=2$), the series above are known to count bipartite graphs on orientable surfaces (resp. surfaces orientable or not) called \textit{maps}, see \cite{GouldenJackson1996a}. Goulden and Jackson have formulated the following two conjectures; see \cite[Conjecture 3.5]{GouldenJackson1996} and \cite[Conjecture 6.2]{GouldenJackson1996}.

\begin{conj}[ Matchings-Jack conjecture \cite{GouldenJackson1996}]\label{conj MJ}
    The coefficients $c^\pi_{\mu,\nu}$ are polynomials in the parameter $b:=\alpha-1$ with non-negative integer coefficients.
\end{conj}

\begin{conj}[\texorpdfstring{$b$}{}-conjecture \cite{GouldenJackson1996}]\label{b-conj}
    The coefficients $h^\pi_{\mu,\nu}$ are polynomials in the parameter $b:=\alpha-1$ with non-negative integer coefficients.
\end{conj}

These two conjectures have combinatorial reformulations in terms of maps counted with ``non-orientability'' weights. As mentioned in the introduction  both of the conjectures are still open. The integrality part in the Matchings-Jack conjecture has been proved in \cite{BenDali2023} and will be useful in the next subsection.
\begin{thm}\cite{BenDali2023}
    For any partitions $\pi,\mu,\nu\vdash n\geq 1$, the coefficient $c^\pi_{\mu,\nu}$ is a polynomial in $b$ with integer coefficients.  
\end{thm}

\subsubsection{Macdonald generalization}\label{ssec generalized conjectures}
We define the coefficients $\bfc^\pi_{\mu,\nu}$ and $\bfh^\pi_{\mu,\nu}$ for partitions $\pi,\mu$ and $\nu$ of the same size by 
\begin{multline}\label{eq def c}
    \sum_{\lambda\in\mathbb Y} u^{|\lambda|} t^{-2n(\lambda)}q^{n(\lambda')}\frac{\JJack_\lambda[X]\JJack_\lambda[Y]\JJack_\lambda[Z]}{\normJt_\lambda}\\
    =\sum_{m\geq 0}\sum_{\pi,\mu,\nu\vdash m} \frac{u^m \bfc^\pi_{\mu,\nu}(\alpha,\gamma)}{z_\pi(q,t)}p_\pi[X]p_\mu[Y]p_\nu[Z],
\end{multline}

and 
\begin{multline}\label{eq def h}
    \log\left(\sum_{\lambda\in\mathbb Y} u^{|\lambda|} t^{-2n(\lambda)}q^{n(\lambda')}\frac{\JJack_\lambda[X]\JJack_\lambda[Y]\JJack_\lambda[Z]}{\normJt_\lambda}\right)\\
    =\sum_{m\geq 0}\sum_{\pi,\mu,\nu\vdash m} \frac{u^m \bfh^\pi_{\mu,\nu}(\alpha,\gamma)}{\alpha [m]_q }p_\pi[X]p_\mu[Y]p_\nu[Z],
\end{multline}
where 
\begin{equation}\label{eq norm j t}
\normJt_\lambda:=\gamma^{-2|\lambda|}\normJ_\lambda=\left\langle\JJack_\lambda,\JJack_\lambda\right\rangle_{q,t},  
\end{equation}
and 
$$[m]_q:=1+q+\dots+q^{m-1}.$$
It is straightforward from the definitions and \cref{eq Macdonald-Jack} that 
$$\bfc^\pi_{\mu,\nu}(\alpha,\gamma=0)=c^\pi_{\mu,\nu}(\alpha)\text{ and }\bfh^\pi_{\mu,\nu}(\alpha,\gamma=0)=h^\pi_{\mu,\nu}(\alpha).$$

\begin{conj}[A Macdonald generalization of the Matchings-Jack conjecture]\label{conj MJ generalized}
For any positive integer $n$ and partitions $\pi,\mu,\nu$ of $n$, the quantity
    $$(1+\gamma)^{n(n-1)}z_\mu z_\nu \bfc^\pi_{\mu,\nu}(\alpha,\gamma)$$
    is a polynomial in $b$ and $\gamma$ with non-negative integer coefficients.
\end{conj}

\begin{conj}[A Macdonald generalization of the \texorpdfstring{$b$}{}-conjecture]\label{b-conj  generalized}
For any positive integer $n$ and partitions $\pi,\mu,\nu$ of $n$, the quantity
    $$(1+\gamma)^{n(n-1)}z_\pi z_\mu z_\nu \bfh^\pi_{\mu,\nu}(\alpha,\gamma)$$
    is a polynomial in $b$ and $\gamma$ with non-negative integer coefficients.
\end{conj}

\cref{conj MJ generalized} has been tested for $n\leq 8$ and \cref{b-conj  generalized} for $n\leq 9$.

\begin{rmq}
In Equations \ref{eq def c} and \cref{eq def h} it seems possible to change the factor $t^{-2n(\lambda)}q^{n(\lambda')}$ by  $t^{-n(\lambda)}(t^{-n(\lambda)}q^{n(\lambda')})^r$ for some $r\geq 0$ and the conjectures above still hold. However, we will prove in \cref{ssec structure coefficients} that for the specific choice of $r=1$, the coefficients  $\bfc^\pi_{\mu,\nu}(\alpha,\gamma)$ are a special case of the structure coefficients of the characters $\MchJack_\mu$. This implies that in some sense these coefficients are a natural two parameters generalization of the coefficients considered by Goulden and Jackson and justifies the factor $t^{-2n(\lambda)}q^{n(\lambda')}$ which appears in the previous definitions.

\end{rmq}

\begin{prop}
    \cref{conj MJ generalized} implies \cref{conj MJ}.
\begin{proof}
     \cref{conj MJ generalized} implies that the coefficients $c^\pi_{\mu,\nu}(\alpha)$ are polynomials in $\alpha$ with positive coefficients. But  these polynomials have integer coefficients by \cite{BenDali2023}. This concludes the proof. 
\end{proof}
\end{prop}

Note also that in a similar way the positivity in \cref{b-conj  generalized} implies the positivity in \cref{b-conj}.

\subsection{Connection with the Structure coefficients \texorpdfstring{$\bfg^\pi_{\mu,\nu}$}{} of Macdonald characters}\label{ssec structure coefficients}
In this subsection, we consider the structure coefficients of Macdonald characters, and we prove that in some sense they generalize the coefficients $\bfc^\pi_{\mu,\nu}$ considered in \cref{ssec generalized conjectures}.

Note that $\MchJack_\mu(\lambda)$ is obtained from $\Mch_\mu$ by a normalization by a scalar and a change of variables (see \cref{eq def theta alpha gamma}), hence their structure coefficients are well defined:
 \begin{equation}\label{eq structure coefficients}
  \MchJack_\mu\MchJack_\nu=\sum_{\pi}\bfg^\pi_{\mu,\nu}(\alpha,\gamma)\MchJack_\pi.   
 \end{equation}

It follows from \cref{prop Jack characters} that the coefficients $\bfg^\pi_{\mu,\nu}(\alpha,\gamma)$ are a two parameter generalization of structure coefficients of Jack characters $\Jch_\mu$ introduced by Do\l{}e\k{}ga and Féray in \cite{DolegaFeray2016} (see also \cite{Sniady2019}):
$$\Jch_\mu\Jch_\nu=\sum_{\pi}\bfg^\pi_{\mu,\nu}(\alpha,\gamma=0)\Jch_\pi.$$

In the following, we will prove that in the case $|\pi|=|\mu|=|\nu|$ the coefficients $\bfg^\pi_{\mu,\nu}(\alpha,\gamma)$ coincide with the coefficients $\bfc^\pi_{\mu,\nu}(\alpha,\gamma)$  defined in \cref{ssec generalized conjectures}. The proof is very similar to the one given in \cite{DolegaFeray2016} for the Jack case.
We start by proving some properties of the coefficients $\bfg^\pi_{\mu,\nu}$.
\begin{lem}\label{lem bounds}
    The coefficient 
    $\bfg^\pi_{\mu,\nu}$ is 0 unless 
    $$\max(|\mu|,|\nu|)\leq|\pi|\leq |\mu|+|\nu|.$$

    \begin{proof}
      The upper bound is a consequence of the fact that $\MchJack_\mu \MchJack_\nu$ is a shifted symmetric function of degree $|\mu|+|\nu|$ and that $\left(\MchJack_\pi\right)_{|\pi|\leq d}$ is a basis of the space of shifted symmetric functions of degree less or equal than $d$. On the other hand, for any partition $\lambda$ such that $|\lambda|<\max(|\mu|,|\nu|)$, one has by the vanishing condition that 
      $\MchJack_\mu(\lambda)\MchJack_\nu(\lambda)=0$, and also that
      $$\sum_{|\pi|\geq \max(\mu,\nu)}\bfg^\pi_{\mu,\nu}(\alpha,\gamma)\MchJack_\pi(\lambda)=0.$$ Combining these two equations with \cref{eq structure coefficients}, we get that 
      $$\sum_{|\pi|<\max(|\mu|,|\nu|)}\bfg^\pi_{\mu,\nu}(\alpha,\gamma)\MchJack_\pi(\lambda)=0,\quad \text{ for any }|\lambda|<\max(|\mu|,|\nu|).$$
      But $\sum_{|\pi|< \max(\mu,\nu)}\bfg^\pi_{\mu,\nu}(\alpha,\gamma)\MchJack_\pi$ is a shifted symmetric function of degree smaller than $\max(|\mu|,|\nu|)$. Using \cref{lem uniqueness}, we deduce that it is identically equal to 0, therefore  $\bfg^\pi_{\mu,\nu}(\alpha,\gamma)=0$ for any $|\pi|<\max(\mu,\nu)$.
    \end{proof}
\end{lem}

We deduce the following corollary.

\begin{cor}\label{cor g}
    Fix a positive integer $m$ and three partitions $\lambda,\mu,\nu\vdash m$. Then 
    $$\MchJack_\mu(\lambda) \MchJack_\nu(\lambda)=\sum_{\pi\vdash m}\bfg^\pi_{\mu,\nu}(\alpha,\gamma)\MchJack_\pi(\lambda).$$ 

    \begin{proof}
        We start by evaluating \cref{eq structure coefficients} in $\lambda$. From the vanishing condition we know that partitions $\pi$ of size larger than $m$ do not contribute to the sum. But  applying \cref{lem bounds} we get that $\bfg^\pi_{\mu,\nu}=0$ if $|\pi|<m$. This concludes the proof. 
    \end{proof}
\end{cor}

We then have the following proposition.

\begin{prop}\label{prop g c}
    Let  $\pi,\mu$ and $\nu$ be three partitions of the same size $m$. Then 
    $$\bfc^\pi_{\mu,\nu}(\alpha,\gamma)=\bfg^\pi_{\mu,\nu}(\alpha,\gamma).$$
    \begin{proof}

We introduce for each $\mu,\nu\vdash m$ the two generating series
$$C_{\mu,\nu}:=\sum_{\pi\vdash m}\frac{\bfc^\pi_{\mu,\nu}(\alpha,\gamma)}{z_\pi(q,t)}p_\pi[X],$$
and 
$$G_{\mu,\nu}:=\sum_{\pi\vdash m}\frac{\bfg^\pi_{\mu,\nu}(\alpha,\gamma)}{z_\pi(q,t)}p_\pi[X].$$
We want to prove that these two series are equal.
From the definition of the coefficients $\bfc^\pi_{\mu,\nu}$ and \cref{eq JJack characters} we have
\begin{align*}
    C_{\mu,\nu}
    =\sum_{\lambda\vdash m}\frac{\MchJack_\mu(\lambda) \MchJack_\nu(\lambda)}{t^{-2n(\lambda)} q^{n(\lambda')}\normJt_\lambda}\JJack_\lambda[X].
\end{align*}
    Using \cref{cor g}, we get that 
    \begin{equation}\label{eq C}
      C_{\mu,\nu}=\sum_{\pi,\lambda\vdash m} \frac{\bfg^\pi_{\mu,\nu}(\alpha,\gamma)\MchJack_\pi(\lambda)}{t^{-2n(\lambda)} q^{n(\lambda')}\normJt_\lambda}\JJack_\lambda[X].
    \end{equation}
    
    On the other hand, using the fact that both Macdonald polynomials and the power-sum functions are orthogonal families, \cref{eq JJack characters} can be inverted as follows
    $$\frac{p_\pi}{z_\pi(q,t)}=\sum_{\lambda\vdash m}\frac{\MchJack_\pi(\lambda)}{t^{-2n(\lambda)}q^{n(\lambda')}\normJt_\lambda}\JJack_\lambda.$$
    Hence, \cref{eq C} becomes 
    $$C_{\mu,\nu}=\sum_{\pi\vdash m}\frac{\bfg^\pi_{\mu,\nu}}{z_\pi(q,t)}p_\pi[X].$$
    
    This is precisely the series $G_{\mu,\nu}$, which concludes the proof of the proposition.
\end{proof}
\end{prop}

Let $f$ be the function defined on tuples of non-negative integers $(n_1,n_2,k)$ by
$$f(n_1,n_2,k):=(M-m)(M+m-k)+m(m-1)-(k-M)(k-M-1),$$
where 
$$M:=\max(n_1,n_2) \text{ and } m=\min(n_1,n_2).$$
We consider the following conjecture.
\begin{conj}\label{conj structure coeff}
Let $\pi,\mu,\nu$ be three partitions.
    Then, the normalized coefficients 
    $$(1+\gamma)^{f(|\mu|,|\nu|,|\pi|)} z_\mu z_\nu \bfg^\pi_{\mu,\nu}$$
    are polynomials in $b:=\alpha-1$ and $\gamma$ with non-negative integer coefficients.
\end{conj}
This conjecture has been tested for $|\pi|,|\mu|,|\nu|\leq 7$. Since $f(n,n,n)=n(n-1)$, and given \cref{prop g c}, it is easy to check that \cref{conj MJ generalized} is a special case of \cref{conj structure coeff}.

\begin{rmq}
    \'{S}niady has formulated a positivity conjecture about the structure coefficients of Jack characters \cite[Conjecture 2.2]{Sniady2019}. This conjecture is related to the case $\gamma=0$ in \cref{conj structure coeff} but the normalizations are different.
\end{rmq}

\subsection{Reformulation with the super nabla operator}\label{sec Reformulation nabla}
The super nabla operator has been recently introduced in \cite{BergeronHaglundIraciRomero2023}. It is defined by its action on modified Macdonald polynomials 
$$\nabla_Y \widetilde H_\lambda[X]=\widetilde H_\lambda[X] \widetilde H_\lambda[Y],$$
where $Y:=y_1+y_2+\dots$ is a second alphabet of variables. We consider here the integral version of this operator $\nablaJ_Y$ defined by 
$$\nablaJ_Y \J_\lambda[X]=t^{-n(\lambda)}\J_\lambda[X] \J_\lambda\left[Y\right].$$
\begin{prop}\label{prop coef c-supernabla}
Let $\pi \vdash m$.
    $$\frac{1}{\gamma^m}\nablaJ\cdot \nablaJ_{Y}\cdot p_\pi[X]=\sum_{\mu,\nu\vdash m}\bfc^{\pi}_{\mu,\nu}(\alpha,\gamma)p_\mu[X] p_\nu[Y].$$
    \begin{proof}
        Let $Z$ be a third alphabet of variables, and let $\widetilde \nablaJ$ and $\widetilde \nablaJ_Y$ denote respectively the nabla and the super nabla operators acting on the space of symmetric functions in the alphabet $Z$.   
        Using the fact that power-sum functions form an orthogonal basis, one has
        \begin{align*}
            \frac{1}{\gamma^m}\nablaJ\nablaJ_{Y}\cdot p_\pi[X]
            &=\left\langle\frac{1}{\gamma^m}\widetilde\nablaJ\widetilde\nablaJ_{Y}\cdot p_\pi[Z],\sum_{\pi\vdash m}\frac{p_\mu[X]p_\mu[Z]}{z_\mu(q,t)}\right\rangle_{q,t},
        \end{align*}
    where the scalar product is taken with respect to the alphabet $Z$. But using the Cauchy identity \cref{eq Cauchy} we get
        \begin{align*}
            \frac{1}{\gamma^m}\nablaJ\nablaJ_{Y}\cdot p_\pi[X]
            &=\left\langle\frac{1}{\gamma^m}\widetilde\nablaJ\widetilde\nablaJ_{Y}\cdot p_\pi[Z],\sum_{\lambda\vdash m}\frac{\J_\lambda[X]\J_\lambda[Z]}{\normJ_\lambda}\right\rangle_{q,t}.
        \end{align*}
     Using the fact that the nabla operators are self-dual, we get 
\begin{align*}
   \frac{1}{\gamma^m}\nablaJ\nablaJ_{Y}\cdot p_\pi[X]
   &= \left\langle\frac{1}{\gamma^m} p_\pi[Z],\widetilde\nablaJ_{Y}\widetilde\nablaJ\cdot \sum_{\lambda\vdash m}\frac{\J_\lambda[X]\J_\lambda[Z]}{\normJ_\lambda}\right\rangle_{q,t}\\
   &=\left\langle\frac{1}{\gamma^m} p_\pi[Z],(-1)^m \sum_{\lambda\vdash m}q^{n(\lambda')}t^{-2n(\lambda)}\frac{\J_\lambda[X]\J_\lambda[Z]\J_\lambda[Y]}{\normJ_\lambda}\right\rangle_{q,t}\\
&=\left\langle p_\pi[Z],\sum_{\lambda\vdash m}q^{n(\lambda')}t^{-2n(\lambda)}\frac{\JJack_\lambda[X]\JJack_\lambda[Z]\JJack_\lambda[Y]}{\normJt_\lambda}\right\rangle_{q,t}\\
&=\sum_{\mu,\nu\vdash m}\bfc^{\pi}_{\mu,\nu}(\alpha,\gamma)p_\mu[X] p_\nu[Y].\qedhere
\end{align*}
    \end{proof}
\end{prop}

\subsection{Special cases in \texorpdfstring{\cref{conj MJ generalized}}{}}\label{sec Special cases}
In this subsection, we discuss some particular cases in \cref{conj MJ generalized}, respectively related to marginals sums and to the specialization $q=t$.

\subsubsection{Marginal sums}
We recall the usual $q$-notation.
We set 
$$[m]_q:=1+q+\dots+q^{m-1} \text{ for any }  m\geq1,$$
$$[0]_q!=1 \text{ and }[m]_q!:=[m]_q[m-1]_q\dots[1]_q,\text{ for any }  m\geq1,$$

$$\text{and } \left[\begin{array}{c}
     m  \\
    k_1,k_2,\dots,k_l 
\end{array}\right]_q:=\frac{[m]_q!}{[k_1]_q!\dots [k_l]_q!}$$
for any $m,l\geq 0$ and $k_1+\dots+k_l=m$.
It is well known that all these quantities are polynomials in $q$ with non-negative integer coefficients. Finally, if $a$ is a parameter then
$$(a;q)_m:=(1-a)(1-qa)\cdots(1-q^{m-1}a) \text{ for any }m\geq 1.$$
\begin{lem}\label{lem J specialized in 1}
Let $\lambda$ be a partition of size $m$. Then,
    $$\sum_{\mu\vdash m}\left[p_\mu\right]\J_\lambda=J_\lambda[1]=\left\{\begin{array}{cc}
         (t;q)_m&\text{ if }\lambda=[m],  \\
         0& otherwise. 
    \end{array}\right.$$
    \begin{proof}
        The first part of the equation is direct from the definitions. To obtain the second one we take $v=t$ in \cref{thm plethystic subt}. 
    \end{proof}
\end{lem}
For a proof of the following lemma, see \cite[Chapter VI, Section 8, Example 1]{Macdonald1995}.
\begin{lem}\label{lem J_m}
    For every $m\geq1$,
    $$\J_{[m]}[X]=(q;q)_m\cdot h_m\left[X\frac{1-t}{1-q}\right]=(q;q)_m\sum_{\pi\vdash n}\frac{1}{z_\pi}p_\pi\left[X\frac{1-t}{1-q}\right].$$
\end{lem}

\begin{prop}\label{prop marginal sums}
Let $\pi,\mu\vdash m.$
    $$\sum_{\nu\vdash m}\bfc^\pi_{\mu,\nu}(\alpha,\gamma)=(1-t)^{-m}\frac{1}{z_\mu}q^{\binom{m}{2}}(q;q)_m p_\mu\left[\frac{1-t}{1-q}\right].$$
    \begin{proof}
    We adapt here the proof given in \cite{GouldenJackson1996} for the Jack polynomials setting. Using \cref{lem J specialized in 1}, we write
    \begin{align*}
  \sum_{\nu\vdash m} \bfc^\pi_{\mu,\nu}(\alpha,\gamma)
  &=(1-t)^{-m}[p_\mu[Y]]\left\langle p_\pi[X],\sum_{\lambda\vdash m}t^{-2n(\lambda)}q^{n(\lambda')}\frac{\J_\lambda[X]\J_\lambda[Y]\J_\lambda[1]}{\normJ_\lambda}\right\rangle_{q,t}.
  \end{align*}
    But we know from \cref{lem J specialized in 1} that only the term corresponding to $\lambda=[m]$ contributes to the sum. Hence,
  \begin{align*}
  \sum_{\nu\vdash m} \bfc^\pi_{\mu,\nu}(\alpha,\gamma)&=(1-t)^{-m}[p_\mu[Y]]\left\langle p_\pi[X],t^{-2n([m])}q^{n([m])}\frac{\J_{[m]}[X]\J_{[m]}[Y]\J_{[m]}[1]}{\normJ_\lambda}\right\rangle_{q,t}\\
  &=(1-t)^{-m}[p_\mu[Y]]\left\langle p_\pi[X],q^{\binom{m}{2}}\frac{\J_{[m]}[X]\J_{[m]}[Y]\cdot(t;q)_m}{(q;q)_m(t;q)_m}\right\rangle_{q,t}.
    \end{align*}
We conclude using \cref{lem J_m}.
    \end{proof}
\end{prop}

\begin{cor}
For any $m\geq 1$ and $\pi,\mu,\nu\vdash m$,
$$z_\mu\sum_{\nu\vdash m} \bfc^\pi_{\mu,\nu}(\alpha,\gamma)$$
     is a polynomial  in $\gamma$ and $\alpha$ with non-negative integer coefficients.
     \begin{proof}
     From \cref{prop marginal sums}, we get
     \begin{align*}
       z_\mu\sum_{\nu\vdash m} \bfc^\pi_{\mu,\nu}(\alpha,\gamma)
       &=\frac{q^{\binom{m}{2}}(q;q)_m }{(1-t)^m}\prod_{1\leq i\leq \ell(\mu)}\frac{1-t^{\mu_i}}{1-q^{\mu_i}}  \\
       &=\frac{q^{\binom{m}{2}}(q;q)_m }{(1-q)^{\ell(\mu)}(1-t)^{m-\ell(\mu)}} \prod_{1\leq i\leq \ell(\mu)}\frac{[\mu_i]_t}{[\mu_i]_q}.
     \end{align*}
     
          But $(q;q)_m=[m]_q!(1-q)^m.$ Hence,
     \begin{align*}
       z_\mu\sum_{\nu\vdash m} \bfc^\pi_{\mu,\nu}(\alpha,\gamma)
       &=q^{\binom{m}{2}}[m]_q! \frac{(1-q)^{m-\ell(\mu)}}{(1-t)^{m-\ell(\mu)}} \prod_{1\leq i\leq \ell(\mu)}\frac{[\mu_i]_t}{[\mu_i]_q}\\
       &=q^{\binom{m}{2}}[m]_q! \alpha^{m-\ell(\mu)}\prod_{1\leq i\leq \ell(\mu)}\frac{[\mu_i]_t}{[\mu_i]_q}.
     \end{align*}

         Note that 
         $[m]!_q\prod_{1\leq i \leq\ell(m)}\frac{1}{[\mu_i]_q}$ is divisible by the binomial 
         $\left[\begin{array}{c}
              m  \\
              \mu_1,\mu_2,\dots,\mu_{\ell(\mu)}
         \end{array}\right]_q.$
         Hence, $[m]!_q\prod_{1\leq i \leq\ell(m)}\frac{1}{[\mu_i]_q}$ is a positive 
 polynomial in $q$, and by consequence in $\alpha$ and $\gamma$. This finishes the proof.
     \end{proof}
\end{cor}

\subsubsection{Integrality for the case $q=t$}

When $q=t$ (equivalently $\alpha=1)$, Macdonald polynomials are Schur functions up to a scalar factor, (see \cite[Chapter VI, Remark 8.4]{Macdonald1995})
$$\mathfrak J^{(\alpha=1,\gamma)}_\lambda=H_\lambda(t) s_\lambda,$$
where $t=\gamma+1$ as usual, and where 
$$H_\lambda(t)=\prod_{\Box\in \lambda} [a_\lambda(\Box)+\ell_\lambda(\Box)+1]_t $$
is a $t$-deformation of the hook product.
Moreover, it follows from \cref{eq j norm} and \cref{eq norm j t} that
$$\widetilde j^{(\alpha=1,\gamma)}_\lambda=H_\lambda(t)^2.$$

We recall that the expansion of Schur functions in the power-sum basis are given by the irreducible characters of the symmetric group $\chi^\lambda$, see e.g \cite[Chapter I]{Macdonald1995}
$$s_\lambda=\sum_{\mu\vdash |\lambda|}\frac{\chi^\lambda(\mu)}{z_\mu}p_\mu.$$

Hence we obtain the following formula for the coefficient of the generalized Matchings-Jack conjecture.
\begin{prop}\label{prop c alpha=1}
    For any partitions $\pi,\mu$ and $\nu$, one has
    $$\bfc^\pi_{\mu,\nu}(\alpha=1,\gamma)=\frac{1}{z_\mu z_\nu} \sum_{\lambda\vdash |\pi|} t^{n(\lambda')-2n(\lambda)}H_\lambda(t)\chi^\lambda(\pi)\chi^\lambda(\mu)\chi^\lambda(\nu).$$
\end{prop}

We deduce the integrality of the coefficients in the parameter $\gamma$.
\begin{cor}
 For any partitions $\pi,\mu$ and $\nu$ of size $m$, the normalized coefficient \\
    $(1+\gamma)^{m(m-1)}z_\mu z_\nu \bfc^\pi_{\mu,\nu}(\alpha=1,\gamma)$ is a polynomial in $\gamma$ with integer coefficients.
    \begin{proof}
        We use \cref{prop c alpha=1} and the fact that $n(\lambda')-2n(\lambda)$ is minimal when $\lambda=[1^m]$ and the corresponding minimum is $m(m-1).$
    \end{proof}
\end{cor}

\bibliographystyle{amsalpha}
\bibliography{biblio.bib}

@preamble{"\def\cprime{$'$}"}

@article {GarsiaTeslerAdvances96,
    AUTHOR = {Garsia, A. M. and Tesler, G.},
     TITLE = {Plethystic formulas for {M}acdonald {$q,t$}-{K}ostka
              coefficients},
   JOURNAL = {Adv. Math.},
  FJOURNAL = {Advances in Mathematics},
    VOLUME = {123},
      YEAR = {1996},
    NUMBER = {2},
     PAGES = {144--222},
      ISSN = {0001-8708,1090-2082},
   MRCLASS = {05E05 (33C80 33D80)},
  MRNUMBER = {1420484},
MRREVIEWER = {Mark\ D.\ Haiman},
       DOI = {10.1006/aima.1996.0071},
       URL = {https://doi.org/10.1006/aima.1996.0071},
}

@article{ChapuyDolega2022,
	author = {Chapuy, Guillaume and Do{\l}{\k{e}}ga, Maciej},
	doi = {10.1016/j.aim.2022.108645},
	fjournal = {Advances in Mathematics},
	issn = {0001-8708},
	journal = {Adv. Math.},
	mrclass = {05E05 (14H57 14N10)},
	mrnumber = {4477016},
	number = {part A},
	pages = {Paper No. 108645, 72},
	title = {Non-orientable branched coverings, {$b$}-{H}urwitz numbers, and positivity for multiparametric {J}ack expansions},
	url = {https://doi.org/10.1016/j.aim.2022.108645},
	volume = {409},
	year = {2022}}

@article {CuencaDolegaMoll2023,
    AUTHOR = {Cuenca, Cesar and Do{\l}\k{e}ga, Maciej and Moll, Alexander},
     TITLE = {Universality of global asymptotics of {J}ack-deformed random
              {Y}oung diagrams at varying temperatures},
   JOURNAL = {Ann. Probab.},
  FJOURNAL = {The Annals of Probability},
    VOLUME = {54},
      YEAR = {2026},
    NUMBER = {1},
     PAGES = {421--488},
      ISSN = {0091-1798,2168-894X},
   MRCLASS = {60F05 (05E05 20C30 60B20 60C05)},
  MRNUMBER = {5019016},
       DOI = {10.1214/25-AOP1769},
       URL = {https://doi.org/10.1214/25-AOP1769},
}

@article{DolegaFeray2016,
	author = {Do{\l}{\k{e}}ga, Maciej and F\'eray, Valentin},
	doi = {10.1215/00127094-3449566},
	fjournal = {Duke Mathematical Journal},
	issn = {0012-7094},
	journal = {Duke Math. J.},
	mrclass = {60C05 (05E05 33C52 60B20)},
	mrnumber = {3498866},
	mrreviewer = {Sho Matsumoto},
	number = {7},
	pages = {1193--1282},
	title = {Gaussian fluctuations of {Y}oung diagrams and structure constants of {J}ack characters},
	url = {http://dx.doi.org/10.1215/00127094-3449566},
	volume = {165},
	year = {2016}}

@article{DolegaFeray2017,
	author = {Do{\l}{\k e}ga, Maciej and F\'{e}ray, Valentin},
	doi = {10.1090/tran/7191},
	fjournal = {Transactions of the American Mathematical Society},
	issn = {0002-9947},
	journal = {Trans. Amer. Math. Soc.},
	mrclass = {05E05},
	mrnumber = {3710651},
	mrreviewer = {Arthur L. B. Yang},
	number = {12},
	pages = {9015--9039},
	title = {Cumulants of {J}ack symmetric functions and the {$b$}-conjecture},
	url = {https://doi.org/10.1090/tran/7191},
	volume = {369},
	year = {2017}}

@article{BenDali2023,
  title={{Integrality in the Matching-{J}ack conjecture and the Farahat-Higman algebra}},
  author={Ben Dali, Houcine},
  journal={Transactions of the American Mathematical Society},
  volume={376},
  number={05},
  pages={3641--3662},
  year={2023}
}

@book{For,
	author = {Forrester, P.J.},
	publisher = {Princeton Univ. Press},
	title = {Log-gases and random matrices},
	year = {2010}}

@article{Sniady2019,
	author = {{\'{S}}niady, Piotr},
	doi = {10.1016/j.jcta.2019.02.020},
	fjournal = {Journal of Combinatorial Theory. Series A},
	issn = {0097-3165},
	journal = {J. Combin. Theory Ser. A},
	mrclass = {05E05 (05E10 46L54)},
	mrnumber = {3921039},
	mrreviewer = {Sho Matsumoto},
	pages = {91--143},
	title = {Asymptotics of {J}ack characters},
	url = {https://doi.org/10.1016/j.jcta.2019.02.020},
	volume = {166},
	year = {2019}}

@article{BenDali2022a,
	author = {Ben Dali, Houcine},
	doi = {10.5802/alco.207},
	journal = {Algebr. Comb.},
	number = {6},
	pages = {1299--1336},
	title = {Generating series of non-oriented constellations and marginal sums in the {M}atching-{J}ack conjecture},
	volume = {5},
	year = {2022}}

@article{CarlssonMellit2018,
	author = {Carlsson, Erik and Mellit, Anton},
	doi = {10.1090/jams/893},
	fjournal = {Journal of the American Mathematical Society},
	issn = {0894-0347},
	journal = {J. Amer. Math. Soc.},
	mrclass = {05E10 (05A30 05E05 33D52)},
	mrnumber = {3787405},
	number = {3},
	pages = {661--697},
	title = {A proof of the shuffle conjecture},
	url = {https://doi.org/10.1090/jams/893},
	volume = {31},
	year = {2018}}

@article{GarsiaHaiman1993,
	author = {Garsia, A. and Haiman, M.},
	doi = {10.1073/pnas.90.8.3607},
	fjournal = {Proceedings of the National Academy of Sciences of the United States of America},
	issn = {0027-8424},
	journal = {Proc. Nat. Acad. Sci. U.S.A.},
	mrclass = {05E05 (05E10 20C30)},
	mrnumber = {1214091},
	mrreviewer = {John R. Stembridge},
	number = {8},
	pages = {3607--3610},
	title = {A graded representation model for {M}acdonald's polynomials},
	url = {http://dx.doi.org/10.1073/pnas.90.8.3607},
	volume = {90},
	year = {1993}}

@article{GouldenJackson1996,
	author = {Goulden, I. P. and Jackson, D. M.},
	coden = {TAMTAM},
	doi = {10.1090/S0002-9947-96-01503-6},
	fjournal = {Transactions of the American Mathematical Society},
	issn = {0002-9947},
	journal = {Trans. Amer. Math. Soc.},
	mrclass = {05E05 (05A15)},
	mrnumber = {1325917 (96m:05196)},
	number = {3},
	pages = {873--892},
	title = {{Connection coefficients, matchings, maps and combinatorial conjectures for {J}ack symmetric functions}},
	url = {http://dx.doi.org/10.1090/S0002-9947-96-01503-6},
	volume = {348},
	year = {1996}}

@article{GouldenJackson1996a,
	author = {Goulden, I. P. and Jackson, D. M.},
	coden = {CJMAAB},
	fjournal = {Canadian Journal of Mathematics. Journal Canadien de Math{\'e}matiques},
	issn = {0008-414X},
	journal = {Canad. J. Math.},
	mrclass = {05C10 (05A15 05E10 57M15)},
	mrnumber = {1402328 (97h:05051)},
	mrreviewer = {Jo\v{z}e Vrabec},
	number = {3},
	pages = {569--584},
	title = {{Maps in locally orientable surfaces, the double coset algebra, and zonal polynomials}},
	volume = {48},
	year = {1996}}

@article{HaglundHaimanLoehr2005,
	author = {Haglund, J. and Haiman, M. and Loehr, N.},
	doi = {10.1090/S0894-0347-05-00485-6},
	fjournal = {Journal of the American Mathematical Society},
	issn = {0894-0347},
	journal = {J. Amer. Math. Soc.},
	mrclass = {05E05},
	mrnumber = {2138143},
	mrreviewer = {Frank Sottile},
	number = {3},
	pages = {735--761},
	title = {A combinatorial formula for {M}acdonald polynomials},
	url = {http://dx.doi.org/10.1090/S0894-0347-05-00485-6},
	volume = {18},
	year = {2005}}

@article{HaglundRemmelWilson2018,
	author = {Haglund, J. and Remmel, J. B. and Wilson, A. T.},
	doi = {10.1090/tran/7096},
	fjournal = {Transactions of the American Mathematical Society},
	issn = {0002-9947},
	journal = {Trans. Amer. Math. Soc.},
	mrclass = {05E05 (05E10)},
	mrnumber = {3811519},
	number = {6},
	pages = {4029--4057},
	title = {The delta conjecture},
	url = {https://doi.org/10.1090/tran/7096},
	volume = {370},
	year = {2018}}

@article{Jack1970/1971,
	author = {Jack, Henry},
	fjournal = {Proceedings of the Royal Society of Edinburgh. Section A. Mathematics},
	issn = {0308-2105},
	journal = {Proc. Roy. Soc. Edinburgh Sect. A},
	mrclass = {12.30 (10.00)},
	mrnumber = {MR0289462 (44 \#6652)},
	mrreviewer = {H. Gupta},
	pages = {1--18},
	title = {{A class of symmetric polynomials with a parameter}},
	volume = {69},
	year = {1970/1971}}

@article{KerovOlshanski1994,
	author = {Kerov, Sergei and Olshanski, Grigori},
	coden = {CASMEI},
	fjournal = {Comptes Rendus de l'Acad{\'e}mie des Sciences. S{\'e}rie I. Math{\'e}matique},
	issn = {0764-4442},
	journal = {C. R. Acad. Sci. Paris S{\'e}r. I Math.},
	mrclass = {05E10 (05E05 20C30 20C32)},
	mrnumber = {MR1288389 (95f:05116)},
	mrreviewer = {Witold Kra{\'s}kiewicz},
	number = {2},
	pages = {121--126},
	title = {{Polynomial functions on the set of {Y}oung diagrams}},
	volume = {319},
	year = {1994}}

@article{KnopSahi1997,
	author = {Knop, Friedrich and Sahi, Siddhartha},
	doi = {10.1007/s002220050134},
	fjournal = {Inventiones Mathematicae},
	issn = {0020-9910},
	journal = {Invent. Math.},
	mrclass = {33D80 (05E05 39A10)},
	mrnumber = {1437493},
	mrreviewer = {Kenji Taniguchi},
	number = {1},
	pages = {9--22},
	title = {A recursion and a combinatorial formula for {J}ack polynomials},
	url = {http://dx.doi.org/10.1007/s002220050134},
	volume = {128},
	year = {1997}}

@article{Lassalle2008b,
	author = {Lassalle, Michel},
	fjournal = {Mathematical Research Letters},
	issn = {1073-2780},
	journal = {Math. Res. Lett.},
	mrclass = {33D52},
	mrnumber = {2424904},
	number = {4},
	pages = {661--681},
	title = {{A positivity conjecture for {J}ack polynomials}},
	volume = {15},
	year = {2008}}

@article{Okounkov1998,
	author = {Okounkov, Andrei},
	coden = {CMPMAF},
	doi = {10.1023/A:1000436921311},
	fjournal = {Compositio Mathematica},
	issn = {0010-437X},
	journal = {Compositio Math.},
	mrclass = {05E05},
	mrnumber = {MR1626029 (99h:05120)},
	mrreviewer = {Ang{\`e}le M. Hamel},
	number = {2},
	pages = {147--182},
	title = {{({S}hifted) {M}acdonald polynomials: {$q$}-integral representation and combinatorial formula}},
	url = {http://dx.doi.org/10.1023/A:1000436921311},
	volume = {112},
	year = {1998}}

@article{Sahi1996,
	author = {Sahi, Siddhartha},
	journal = {Internat. Math. Res. Notices},
	number = {10},
	pages = {457--471},
	publisher = {Oxford University Press},
	title = {Interpolation, integrality, and a generalization of {M}acdonald's polynomials},
	volume = {1996},
	year = {1996}}

@article{Stanley1989,
	author = {Stanley, Richard P.},
	coden = {ADMTA4},
	doi = {10.1016/0001-8708(89)90015-7},
	fjournal = {Advances in Mathematics},
	issn = {0001-8708},
	journal = {Adv. Math.},
	mrclass = {05A15 (20C30)},
	mrnumber = {1014073 (90g:05020)},
	mrreviewer = {Dennis White},
	number = {1},
	pages = {76--115},
	title = {{Some combinatorial properties of {J}ack symmetric functions}},
	url = {http://dx.doi.org/10.1016/0001-8708(89)90015-7},
	volume = {77},
	year = {1989}}

@book{Macdonald1995,
	address = {New York},
	author = {Macdonald, I. G.},
	edition = {Second},
	isbn = {0-19-853489-2},
	mrclass = {05E05 (05-02 20C30 20C33 20K01 33C80 33D80)},
	mrnumber = {1354144},
	mrreviewer = {John R. Stembridge},
	note = {With contributions by A. Zelevinsky, Oxford Science Publications},
	pages = {x+475},
	publisher = {The Clarendon Press Oxford University Press},
	series = {{Oxford Mathematical Monographs}},
	title = {{Symmetric functions and {H}all polynomials}},
	year = {1995}}

@unpublished{BenDaliDolega2023,
  title={{Positive formula for Jack polynomials, {J}ack characters and proof of Lassalle's conjecture}},
  author={Ben Dali, Houcine and Do{\l}{\k{e}}ga, Maciej},
  note={preprint arXiv:2305.07966},
  year={(2023)}
}

@article{Lassalle1998,
  title={Coefficients binomiaux g{\'e}n{\'e}ralis{\'e}s et polyn{\^o}mes de {M}acdonald},
  author={Lassalle, Michel},
  journal={J. Funct. Anal.},
  volume={158},
  number={2},
  pages={289--324},
  year={1998},
  publisher={Academic Press}
}

@unpublished{BergeronHaglundIraciRomero2023,
  title={The super {n}abla operator},
  author={Bergeron, Fran{\c{c}}ois and Haglund, Jim and Iraci, Alessandro and Romero, Marino},
  note={preprint arXiv:2303.00560},
  year={(2023)}
}

@article{GarsiaMellit2019,
  title={Five-term relation and {M}acdonald polynomials},
  author={Garsia, Adriano and Mellit, Anton},
  journal={Journal of Combinatorial Theory, Series A},
  volume={163},
  pages={182--194},
  year={2019},
  publisher={Elsevier}
}

@inproceedings{GarsiaHaimanTesler2001,
  title={Explicit plethystic formulas for {M}acdonald q, t-{K}ostka coefficients},
  author={Garsia, Adriano M and Haiman, Mark and Tesler, Glenn},
  booktitle={The Andrews Festschrift: Seventeen Papers on Classical Number Theory and Combinatorics},
  pages={253--297},
  year={2001},
  organization={Springer}
}

@article{DAdderioMellit2022,
  title={A proof of the compositional {D}elta conjecture},
  author={D'Adderio, Michele and Mellit, Anton},
  journal={Advances in Mathematics},
  volume={402},
  pages={108342},
  year={2022},
  publisher={Elsevier}
}

@article{DAdderioRomero2023,
  title={New identities for {T}heta operators},
  author={D’Adderio, Michele and Romero, Marino},
    Volume = {376},
    Number = {8},
    Pages = {5775--5807},
  journal={Transactions of the American Mathematical Society},
  DOI = {10.1090/tran/8911},
  year={2023}
}

@article {DAdderioIraciVWyngaerd21,
    AUTHOR = {D'Adderio, Michele and Iraci, Alessandro and Vanden Wyngaerd,
              Anna},
     TITLE = {Theta operators, refined delta conjectures, and coinvariants},
   JOURNAL = {Adv. Math.},
  FJOURNAL = {Advances in Mathematics},
    VOLUME = {376},
      YEAR = {2021},
     PAGES = {Paper No. 107447, 59},
      ISSN = {0001-8708,1090-2082},
   MRCLASS = {05E05},
  MRNUMBER = {4178919},
MRREVIEWER = {Elizabeth\ M.\ Niese},
       DOI = {10.1016/j.aim.2020.107447},
       URL = {https://doi.org/10.1016/j.aim.2020.107447},
}

@article{BergeronGarsiaHaimanTesler1999,
  title={Identities and positivity conjectures for some remarkable operators in the theory of symmetric functions},
  author={Bergeron, Fran{\c{c}}ois and Garsia, Adriano M and Haiman, Mark and Tesler, Glenn},
  journal={Methods and applications of analysis},
  volume={6},
  number={3},
  pages={363--420},
  year={1999},
  publisher={International Press of Boston}
}

@article{Haiman02,
  title={Vanishing theorems and character formulas for the Hilbert scheme of points in the plane},
  author={Haiman, Mark},
  journal={Inventiones mathematicae},
  volume={149},
  pages={371--407},
  year={2002},
  publisher={Springer}
}

@incollection {AlexanderssonFeray19,
    AUTHOR = {Alexandersson, Per and F\'eray, Valentin},
     TITLE = {A positivity conjecture on the structure constants of shifted
              {J}ack functions},
 BOOKTITLE = {Open problems in algebraic combinatorics},
    SERIES = {Proc. Sympos. Pure Math.},
    VOLUME = {110},
     PAGES = {51--59},
 PUBLISHER = {Amer. Math. Soc., Providence, RI},
      YEAR = {2024},
      ISBN = {[9781470473334]; [9781470477974]},
   MRCLASS = {05E05 (20C30)},
  MRNUMBER = {4780722},
}

@Article{Knop1997b,
 Author = {Knop, Friedrich},
 Title = {Symmetric and non-symmetric quantum {Capelli} polynomials},
 FJournal = {Commentarii Mathematici Helvetici},
 Journal = {Comment. Math. Helv.},
 ISSN = {0010-2571},
 Volume = {72},
 Number = {1},
 Pages = {84--100},
 Year = {1997},
 DOI = {10.4171/CMH/72.1.7},
}
   
\end{document}